\documentclass{article}
\usepackage[utf8]{inputenc}
\usepackage{amsmath}
\usepackage{amssymb}
\usepackage{amsthm}
\usepackage{enumerate}
\usepackage{cases}
\usepackage{tikz}
\usepackage[ruled,linesnumbered]{algorithm2e}
\usepackage[top=2.5cm,bottom=2.5cm,left=3cm,right=3cm]{geometry}
\usepackage{multirow}
\usepackage{multicol}
\usepackage{makecell}
\usepackage{subfigure}
\usepackage{bm}
\usepackage{float}

\usepackage{hyperref}
\hypersetup{
    colorlinks=true,
    linkcolor=blue,
    filecolor=magenta,      
    urlcolor=blue,
    citecolor=blue,
    }

\newcommand{\Smnp}{{\mathcal{S}_M(n,p)}}

\newcommand{\Rnp}{{\mathbb{R}^{n\times p}}}

\newcommand{\tp}{^\top}
\newcommand{\inner}[1]{\left\langle #1 \right\rangle}
\newcommand{\norm}[1]{\left\Vert #1\right\Vert}
\newcommand{\grad }{\mathrm{grad}}
\newcommand{\tr}{\operatorname{tr}}
\newcommand{\st}{\mathrm{s.~t.}}

\usetikzlibrary{intersections}
\newtheorem{theorem}{Theorem}[section]
\newtheorem{definition}[theorem]{Definition}
\newtheorem{proposition}[theorem]{Proposition} 
\newtheorem{assumption}[theorem]{Assumption}

\newtheorem{problem}[theorem]{Problem}

\newtheorem{corollary}[theorem]{Corollary}

\begin{document}
     \title{A Smooth Locally Exact Penalty Method for Optimization Problems over Generalized Stiefel Manifolds}
     \author{Linshuo Jiang, Xin Liu, Nachuan Xiao}
     	\maketitle

\begin{abstract}
In this paper, we consider a class of optimization problems constrained to the generalized Stiefel manifold. Such problems are fundamental to a wide range of real-world applications, including generalized canonical correlation analysis, linear discriminant analysis, and electronic structure calculations. Existing works mainly focuses on cases where the generalized orthogonality constraint is induced by a symmetric positive definite matrix $M$, a setting where the geometry essentially reduces to that of the standard Stiefel manifold. However, many practical scenarios involve a singular $M$, which introduces significant analytical and computational challenges. Therefore, we propose a \textbf{S}mooth \textbf{L}ocally \textbf{E}xact \textbf{P}enalty model \eqref{slep} and establish its equivalence to the original problem in the aspect of stationary points under a finitly large penalty parameter. This penalty model admits the direct application of various unconstrained optimization techniques, with convergence guarantees inherited from established results. Compared to Riemannian optimization approaches, our proposed penalty mode eliminates the need for retractions and vector transports, hence significantly reducing per-iteration computational costs. Extensive numerical experiments validate our theoretical results and demonstrate the effectiveness and practical potential of the proposed penalty model \ref{slep}.

      \end{abstract}

 \section{Introduction}\label{intro}
     In this paper, we focus on a class of matrix optimization problems with generalized orthogonal constraints, formulated as follows,
    	\begin{equation}\tag{OCP}\label{ocp}
    	\begin{aligned}
    		\min_{X\in\mathbb{R}^{n\times p}}\;\;\; &f(X),\\
    		\st\;\;\; &X^\top M X=I_p.
    	\end{aligned}
    \end{equation}
Here, $M \in \mathbb{R}^{n\times n}$ is a positive semi-definite matrix, $I_p$ denotes the $p$-th order identity matrix with $p\leq n$, and $f:\mathbb{R}^{n\times p}\to \mathbb{R}$ is a function satisfying the following blanket assumption:
\begin{assumption}[The blanket assumption]\label{as:blanket}
    $f$ is locally Lipschitz continuous and smooth on $\mathbb{R}^{n\times p}$.
\end{assumption}
The feasible region of \ref{ocp}, denoted as $\mathcal{S}_M(n,p):=\{X\in\mathbb{R}^{n\times p}:X^\top M X=I_p\}$, is an embedded submanifold of $\mathbb{R}^{n\times p}$, commonly referred to as the generalized Stiefel manifold \cite{absil_optimization_2008,boumal_manopt_2014,shustin_preconditioned_2021,sato_cholesky_2019}.
Let $r:=\operatorname{rank}(M)$. It is straightforward to verify that  $\mathcal{S}_M(n,p)$ is an empty set whenever $r< p$. Therefore, this study focuses on cases where $r \geq p$. 

The optimization problem \ref{ocp} has numerous applications in scientific computing,
data science and statistics \cite{bendory_nonconvex_2016,kasai_riemannian_2018,gao_orthogonalization-free_2022}, including  Generalized Canonical Correlation Analysis (GCCA) \cite{kettenring_canonical_1971,sorensen_generalized_2021,gao_sparse_2021}, Linear Discriminant Analysis (LDA) \cite{chen_sparse_2013,luo_linear_2011}, and  electronic structure calculations \cite{gao_orthogonalization-free_2022}. In the following, we show that GCCA can be formulated as \ref{ocp} and clarify the correspondence between variables and constraints.

\begin{problem}[GCCA]\label{pr:1}
Given $k$ random vectors $\{\mathcal{Z}_i\in\mathbb{R}^{n_i}: 1\leq i\leq k\}$ with $k\ge 2$ and $n:=\sum_{i=1}^k n_i$. The goal of GCCA is to investigate the linear correlations among these $k$ random vectors.  \cite{gao_sparse_2021} proposed the following optimization formulation
for GCCA:
\begin{equation}\label{p:gca}
    \begin{aligned}   
		&&\max_{X\in\mathbb{R}^{n\times p}}\;\;\;&\operatorname{tr}(X^\top\Sigma X),\\
   &&\st\;\;\;& X^\top \Sigma_0 X=I_p,
   \end{aligned}
\end{equation}
   where
\[
   \Sigma:=\mathbb{E}\left[\begin{pmatrix}
\mathcal{Z}_1\mathcal{Z}_1^\top&\dots&\mathcal{Z}_1\mathcal{Z}_1^\top\\ 
	  \vdots &&\vdots\\
\mathcal{Z}_1\mathcal{Z}_1^\top&\dots&\mathcal{Z}_k\mathcal{Z}_k\\
   \end{pmatrix}\right],\quad  \Sigma_0:=\mathbb{E}\left[\begin{pmatrix}
\mathcal{Z}_1\mathcal{Z}_1^\top&&0\\ 
	   &\ddots&\\
0&&\mathcal{Z}_k\mathcal{Z}_k\\
   \end{pmatrix}\right].
\]

 The population covariance matrices $\Sigma$ and $\Sigma_0$ are defined on the underlying distributions and are therefore unobservable in practice. They are typically estimated from finite samples of the corresponding random vectors. In high-dimensional settings—such as genetics, brain imaging, spectroscopic imaging, and climate studies—where the number of features is comparable to or exceeds the sample size, the empirical covariance matrix becomes rank-deficient and non-invertible, posing challenges for subsequent statistical or optimization procedures.
\end{problem}

\begin{problem}[Smoothing approximation subproblem of sparse GCCA]
	To enhance interpretability and practicality, sparse structural assumptions on the solution have been widely adopted in both theoretical and empirical studies, motivating the sparse generalized correlation analysis (SGCCA) \cite{chen_sparse_2013,gao_sparse_2017,gao_sparse_2021}.
The optimization problem for the $l_{2,1}$-norm sparse GCCA can be formulated as
\begin{equation}\label{sgca}
      \begin{aligned}
&&\max_{X\in\mathbb{R}^{n\times p}}\;\;\;&\operatorname{tr}(X^\top\Sigma X)+\gamma \|X\|_{2,1},\\
	   &&\st\;\;\;& X^\top \Sigma_0 X=I_p,
	   \end{aligned}
\end{equation}
   where $\|X\|_{2,1}:=\sum_{i=1}^{n}\|X_{i\cdot}\|_2$ for $X\in\mathbb{R}^{n\times p}$, and $\gamma$ is the sparsity parameter. This is a non-smooth optimization problem over the generalized Stiefel manifold, which can be addressed via a series of smoothing approximation problems \cite{zhang_riemannian_2021}.  A suitable smoothing approximation for the $l_{2,1}$-norm in SGCCA is given by 
\[
	   \begin{aligned}	&&\min_{X\in\mathbb{R}^{n\times p}}\;\;\;&-\frac{1}{2}\operatorname{tr}(X^\top\Sigma X)+\gamma \sum_{i=1}^{p} s_{\mu}(\|X_{i\cdot}\|_2),\\
	   &&\st\;\;\;& X^\top \Sigma_0 X=I_p,
	   \end{aligned}
\]
   where the smoothing function is defined as 
    \begin{equation}\label{smoothfun}
        s_\mu(t):=\left\{\begin{aligned}
	   &|t|,\;\;\;&\text{if } |t|>\frac{\mu}{2},\\ 
	   &\frac{t^2}{\mu}+\frac{\mu}{4},\;\;\;&\text{if } |t|\le\frac{\mu}{2},
   \end{aligned}\right.
    \end{equation}
   and
   $\mu$ denotes the smoothing parameter.
  
\end{problem}

\subsection{Existing works}\label{existwork}
A large number of efficient algorithms for solving \ref{ocp} fall into the category of Riemannian optimization approaches \cite{absil_optimization_2008,boumal2023introduction}, including Riemannian gradient methods \cite{abrudan_steepest_2008,iannazzo_r_2018}, Riemannian conjugate gradient methods \cite{sato_new_2015,sato_riemannian_2022}, Riemannian trust-region methods \cite{absil_trust_2007,baker_riemannian_2008}, Riemannian Newton methods \cite{adler_newton_2002}, and Riemannian quasi-Newton methods \cite{huang_broyden_2015}. Interested readers can refer to \cite{absil_optimization_2008,boumal2023introduction} and the recent survey \cite{hu2020brief} for further details. These methods are derived by extending unconstrained optimization techniques through
geometric tools such as Riemannian gradients, retractions, vector transports,  and logarithmic maps \cite{petersen2006riemannian}.
Therefore, the efficient computation of these geometric tools is crucial for developing practical Riemannian optimization methods. For the generalized Stiefel manifold where $M$ is positive definite, prior studies \cite{yger_adaptive_2012,boumal_manopt_2014,sato_cholesky_2019,shustin_preconditioned_2021} have provided explicit formulations for computing these geometric tools. However, geometric tools for $\mathcal{S}_M(n,p)$ with rank-deficient $M$ remain largely unexplored. Moreover, the use of such tools introduces additional computational costs at each iteration
and reduces the parallelizability of Riemannian optimization algorithms owing to the inclusion
of matrix decomposition operations, as discussed in Section 4.

In addition to Riemannian optimization approaches, other studies have employed splitting methods \cite{zhu_nonconvex_2017,lai_splitting_2014,chen_augmented_2016}, which, however, lack convergence guarantees owing to the nonconvex nature of \ref{ocp}. By contrast, augmented Lagrangian methods can be applied to solve \ref{ocp} with convergence guarantees. The corresponding augmented Lagrangian penalty function is expressed as
\begin{equation}\label{eq:lagrange}
    \mathcal{L}_{\beta}(X,\Lambda):=f(X)-\frac{1}{2}\left\langle\Lambda,X^\top MX-I_p\right\rangle+\frac{\beta}{4}\left\|X^\top MX-I_p\right\|_F^2,
\end{equation}
where $\Lambda\in\mathbb{S}^p:=\{W\in\mathbb{R}^{p\times p}:W^\top=W\}$ denotes the dual variable. However, augmented Lagrangian methods are often inefficient because their performance is highly sensitive to the penalty parameter $\beta$ in nonconvex optimization problems. 

\subsection{Motivation}\label{moti}
In this study, we focus on developing efficient infeasible methods for \ref{ocp}. 
According to the Karush-Kuhn-Tucker (KKT) conditions \cite{nocedal_numerical_1999} for \ref{ocp},
\begin{equation}\label{kkt}
    \left\lbrace \begin{aligned}
		&\nabla f(X^*)-MX^*\Lambda^*= 0,\\
		&{X^*}^\top M X^*=I_p,
	\end{aligned}\right.
\end{equation}
we observe that the Lagrangian multiplier $\Lambda^*$ of the KKT pair $(X^*,\Lambda^*)$ has the closed form
$\Lambda^*={X^*}^\top\nabla f(X^*)\in\mathbb{S}^p$. Substituting the dual variable $\Lambda$
in \eqref{eq:lagrange}  
with the mapping $\Lambda(X):=\Psi(X^\top\nabla f(X))$, we obtain Fletcher's penalty function \cite{fletcher1970class}: 
\begin{equation}\label{eq:laglambda}
\begin{aligned}
h_{FL}(X) :={}& f(X)-\frac{1}{2}\left\langle\Psi(X^\top \nabla f(X)),X^\top MX-I_p\right\rangle+\frac{\beta}{4}\left\|X^\top MX-I_p\right\|_F^2\\
={}&f(X)-\frac{1}{2}\langle \nabla f(X),X(X^\top MX-I_p) \rangle +\frac{\beta}{4}\left\|X^\top MX-I_p\right\|_F^2
\end{aligned}
\end{equation}
where $\Psi(W):=\frac{1}{2}(W + W^\top), W\in\mathbb{R}^{n\times n}$, is the symmetrization operator.
However, including $\nabla f(X)$ in \eqref{eq:laglambda} undermines the differentiability of $h_{FL}$. More importantly, evaluating the function and gradient of $f$ requires computing both $\nabla f$ and $\nabla^2 f$. Therefore, it is challenging to directly employ unconstrained optimization approaches to solve \eqref{eq:laglambda}. 

To overcome these challenges, we approximate $f(X) - \tfrac{1}{2}\langle \nabla f(X), X(X^\top M X - I_p) \rangle$ using a Taylor expansion along the direction $-X(X^\top M X - I_p)$, which yields
$f(X) - \tfrac{1}{2}\langle \nabla f(X), X(X^\top M X - I_p) \rangle= f(X(\tfrac{3}{2}I_p - \tfrac{1}{2}X^\top M X))+\mathcal{O}(\|X^\top M X-I_p\|^2)$.
For convenience, we define
\begin{equation}
\mathcal{A}(X) := X\left(\tfrac{3}{2}I_p - \tfrac{1}{2}X^\top M X\right),
\end{equation}
and thus obtain the following Smooth Locally Exact Penalty (SLEP) model for \ref{ocp}:
\begin{equation}\label{slep}\tag{SLEP}
\min_{X\in\mathbb{R}^{n\times p}}
h(X) := f(\mathcal{A}(X))
+ \tfrac{\beta}{4}\|X^\top M X - I_p\|_F^2.
\end{equation}


\subsection{Contributions}\label{contri}

In this paper, we propose a novel penalty model, \ref{slep}, for the optimization problem over a generalized Stiefel manifold with a possibly rank-deficient matrix $M$ in \ref{ocp}. We prove that, for any feasible point, there exists a neighborhood in which the first- and second-order stationary points of \ref{slep} and \ref{ocp} coincide when the penalty parameter is sufficiently large.


 Based on the smoothness and exactness of SLEP, we can directly employ existing unconstrained optimization algorithms to efficiently solve the OCP through SLEP. 
As an example of this infeasible framework, we introduce the SLBB algorithm, which applies a gradient-based method to \ref{slep}. We also conduct numerical experiments to determine default settings for our algorithms and show that they outperform corresponding Riemannian methods. These results confirm that the proposed \ref{slep} model enables the efficient application of unconstrained optimization methods and offers substantial performance advantages.

\subsection{Organization}\label{org}

The remainder of this paper is organized as follows. Section \ref{pre} introduces preliminaries, including optimality conditions and several useful lemmas. Section \ref{theory} analyzes the properties of the penalty function \ref{slep}, focusing on its exactness and smoothness. Section \ref{algo} presents the proposed algorithm, and Section \ref{numerical} presents the  numerical performance of our proposed algorithm. Finally, Section \ref{conclu} summarizes the main findings.

\section{Preliminaries}\label{pre}
This section first presents the fundamental notations and constants used in the theoretical analysis. It then revisits the first- and second-order optimality conditions of \ref{ocp}. 

\subsection{Notations}\label{notation}
Throughout this paper, we adopt the following notations. 
Given a square matrix $W$, $\operatorname{tr}(W)$ denotes its trace.
 For a symmetric matrix $W$,
$\lambda_{\max}(W)$ and $\lambda_{\min}(W)$ represent its largest and smallest eigenvalues, respectively, and $\lambda_{+,\min}(W)$ denotes the smallest positive eigenvalue.
For any matrix $X$, $\sigma_{\max}(X)$ and $\sigma_{\min}(X)$ represent its largest and smallest singular values, respectively. 
Given two matrices $X,Y$ of the same size, 
$\langle X,Y\rangle:=\operatorname{tr}(X^\top Y)$ represents their canonical inner product. For a symmetric positive-definite matrix $B$ of compatible dimension, we
denote the inner product of $X,Y$ deduced by $B$ as 
$\langle X,Y\rangle_B:=\operatorname{tr}(X^\top B Y)$. 
 The 2-norm and Frobenius norm of a given matrix $X$ are denoted by $\|X\|_2$ and $\|X\|_F$, respectively. 
For a matrix $X$ and positive constant $\rho>0$,
$\mathbb{B}(X,\rho)$ represents the closed ball centered at $X$ with radius $\rho$. 
$\mathbb{S}^n$ denotes the set of all symmetric matrices in $\mathbb{R}^{n\times n}$. 
Given a square matrix $W$, we define $\Psi(W):= \frac{1}{2}(W + W^\top)$ as its symmetric part and $\operatorname{diag}(W)$ as the vector of its diagonal entries.
Moreover,
$\operatorname{Diag}(v)$ denotes the diagonal matrix with all entries of a given vector $v$ along its diagonal.  $\operatorname{qf}(X)$ denotes the Q factor of the QR factorization of matrix $X$.
The set $\operatorname{Null}(M):=\{X\in\mathbb{R}^{n\times p}:MX=0\}$ denotes the null subspace of the matrix $M$ acting on $\Rnp$. Let $\mathcal{P}_1:\mathbb{R}^{n\times p}\to(\operatorname{Null}(M))^{\perp}$ and $\mathcal{P}_2:\mathbb{R}^{n\times p}\to\operatorname{Null}(M)$ denote the orthogonal projections onto $(\operatorname{Null}(M))^{\perp}$ and $\operatorname{Null}(M)$, respectively.
The tangent space of the generalized Stiefel manifold $\mathcal{S}_M(n,p)$ at point $X\in\mathcal{S}_M(n,p)$ is defined as $\mathcal{T}_X:= \left\{D\in\mathbb{R}^{n\times p}| \Psi(D^\top M X) = 0\right\}.$
Let the eigenvalue decomposition of the positive semidefinite matrix $M$ be 
\[M=V\begin{pmatrix}
\Sigma&0\\
0&0
\end{pmatrix}V^\top,\]
where $V \in \mathbb{R}^{n \times n}$ is an orthogonal matrix and $\Sigma \in \mathbb{R}^{r \times r}$ is a diagonal matrix with strictly positive entries.
We then construct a positive-definite matrix
\[\bar{M}:=V\begin{pmatrix}
\Sigma&0\\
0&I_{n-r}
\end{pmatrix}V^\top.\]
The generalized Stiefel manifold $\Smnp$, equipped with the Riemannian metric inherited from the inner product 
$\langle \cdot,\cdot\rangle_{\bar{M}}$, admits the following geometric structure. The normal space at point $X\in\Smnp$ is defined as:
\[\mathcal{N}_X:=\{\mathcal{P}_1(X)S:S\in\mathbb{S}^{p}\}.\]
The Riemannian gradient of the smooth function $f$ is given by
\begin{equation}\label{grad}
    \operatorname{grad}f(X):={\bar{M}}^{-1}\nabla f(X)-\mathcal{P}_1(X)\Psi(\mathcal{P}_1(X)^\top\nabla f(X)),
\end{equation}
and the Riemannian Hessian operator along $T\in\mathcal{T}_X$ is 
\begin{equation}\label{hess}
    \operatorname{hess} f(X)[T]:={\bar{M}}^{-1}(\nabla^2 f(X)[T]-MT\Psi(\mathcal{P}_1(X)^\top\nabla f(X))).
\end{equation}
For any point $X\in\{Y\in\mathbb{R}^{n\times p}:Y^\top MY\succ0\}$, the mapping
$\mathcal{R}(X):=X(X^\top MX)^{-\frac{1}{2}}$
represents the projection onto the generalized Stiefel manifold.

For simplicity, we denote the composition of the objective function $f$ with the operator $\mathcal{A}$ as
\begin{equation}
    g(X):=f(\mathcal{A}(X)),
\end{equation}
and denote the composition of $\nabla f$ with $\mathcal{A}$ as
\begin{align}
     G(X):=\nabla f(\mathcal{A}(X)).
\end{align}
\subsection{Constants}\label{constant}
Under Assumption \ref{as:blanket}, for any $X\in\mathcal{S}_M(n,p)$, we define the following constants,
    \[\begin{aligned}
       H_{X,\nabla f}&:=\sup_{Y\in  \mathbb{B}(X,1)} \|\nabla f(Y)\|_F,
     &&H_{X,\nabla g}:=\sup_{Y\in \mathbb{B}(X,1)} \|\nabla g(Y)\|_F,\\
		H_{X,G}&:=\sup_{Y\in \mathbb{B}(X,1)} \|G(Y)\|_F,
&&L_{X,\nabla g}:=\sup_{X_1,X_2\in \mathbb{B}(X,1)} \frac{\|\nabla g(X_1)-\nabla g(X_2)\|_F}{\|X_1-X_2\|_F},\\
    \end{aligned}\]
where $g(X):=f\left(\mathcal{A}(X)\right), G(X):=\nabla f\left(\mathcal{A}(X)\right).$

In Section \ref{relation}, we establish the equivalence between \ref{ocp} and \ref{slep} in the neighborhood $\mathbb{B}(X,\rho_X)$ of a feasible point $X\in\Smnp$, where the radius $\rho_X$ is defined as
\[\rho_X:=\min\left\{\frac{1}{2},\sup\left\{\rho>0:\sup_{Y\in\mathbb{B}(X,\rho)}\left\|Y^\top MY-I_p\right\|_2\le\frac{1}{6C_X}\right\}\right\},\]
and $C_X:=\|X\|_F+1$. Furthermore, throughout our theoretical analysis, we denote the lower bound of the penalty parameter $\beta$ as 
\[ \beta_{X}:=\max\left\{6H_{X,\nabla f}C_X,6H_{X,G},4L_{X,\nabla g}C_X^2,\frac{6\lambda_{\min}^{\frac{1}{2}}(\bar{M})H_{X,\nabla g}C_X}{\kappa^{\frac{1}{2}}(\bar{M})(1+\lambda_{\max}C_X^2)}\right\}.\]

\subsection{Optimality conditions}\label{opcon}
In this subsection, we present the optimality conditions for both the original Riemannian optimization problem \ref{ocp} and the unconstrained optimization problem \ref{slep}. 
For the original Riemannian optimization problem \ref{ocp}, we define its first- and second-order stationary points as follows:
\begin{definition}[\cite{absil_optimization_2008}]\label{def:o_ocp}
We call $X^*\in\Smnp$ a first-order stationary point of \ref{ocp} if  $\grad _{\bar{M}} f(X^*)=0$. We call $X^*$ a second-order stationary point of \ref{ocp} if $X^*$ is a first-order stationary point of \ref{ocp} and
	\begin{equation}
		\left\langle T,\operatorname{hess}_{\bar{M}}(X^*)[T]\right\rangle_{\bar{M}}\ge0,\;\;\; \forall
  T\in\mathcal{T}_{X^*},
	\end{equation}
	where $\mathcal{T}_{X^*}$ denotes the tangent space of $\Smnp$ at $X^*$.
\end{definition}
For the unconstrained optimization problem \ref{slep}, following \cite{nocedal_numerical_1999}, we define the corresponding optimality conditions as follows.
\begin{definition}[\cite{nocedal_numerical_1999}]\label{def:o_slep}
    We call $X^*\in\Rnp$ a first-order stationary point of  \ref{slep} if $\nabla h(X^*)=0$, and call $X^*\in\Rnp$ a second-order stationary point of \ref{slep} if $X^*$ is a first-order stationary point of \ref{slep} and for any $D\in\Rnp$,  
    \[\inner{D,\nabla^2 h(X^*)[D]}\ge0.\]
Moreover, we call $X^*$ a $c$ a strict saddle point of \ref{slep} if it is a first-order stationary point of \ref{slep}, and there exist constants $c> 0$ and $D\in\Rnp$ such that
\[\inner{D,\nabla^2 h(X^*)[D]}\le-c\norm{D}^2.\]
\end{definition}

\section{Theoretical results}\label{theory}
In this section, we establish the equivalence between \ref{ocp} and \ref{slep}. In particular, we show that these two problems share the same first- and second-order stationary points in the neighborhood of the generalized Stiefel manifold. 
\subsection{Smoothness of \ref{slep}}\label{preslep} 
We first present Proposition \ref{Prop_grad_h}, which provides the closed-form expression for $\nabla h$. The result in Proposition \ref{Prop_grad_h} follows directly from computation; therefore, the proof is omitted for brevity. 
\begin{proposition}
    \label{Prop_grad_h}
    For every $X\in\mathbb{R}^{n\times p}$, we have
    \begin{equation}\label{eq:grad_h}
        \nabla h(X)=G(X)\left(\frac{3}{2}I_p-\frac{1}{2}X^\top MX\right)-MX\Psi\left(X^\top G(X)\right)+\beta MX\left(X^\top MX-I_p\right).
    \end{equation}
\end{proposition}

Based on the formulation of $\nabla h$ in Proposition \ref{Prop_grad_h}, we summarize the computational complexity of evaluating $\nabla h(X)$ in Table \ref{ta:comput_slepgra}. The computation of $\nabla h(X)$ involves only the evaluation of $\nabla f$ and standard matrix-matrix multiplications, making it computationally efficiently in practice. 
\begin{table}[htbp]
	\centering 
	\begin{tabular}{c|c|c}
		\hline\hline
		\multirow{5}{*}{$\nabla G(X)(\frac{3}{2}I_p-\frac{1}{2}X^\top MX)$}&$MX$&$2n^2p$\\
		&$X^\top M X$&$np^2$\\
		&$X(X^\top MX-I_p)$&$2np^2$\\
		&$G(X)=\nabla f(Y)|_{Y=X(\frac{3}{2}I_p-\frac{1}{2}X^\top M X)}$&$1FO$\\
		&$G(X)(\frac{3}{2}I_p-\frac{1}{2}X^\top M X)$&$2np^2$\\
		\hline
		\multirow{2}{*}{$MX\Psi(X^\top G(X))$}&$X^\top G(X)$&$2np^2$\\
		&$MX\Psi(X^\top G(X))$&$2np^2$\\
		\hline
		\multicolumn{2}{c|}{$MX(X^\top MX-I_p)$}&$2np^2$\\
		\hline
		\multicolumn{2}{c|}{In total}&$1FO+2n^2p+11np^2$\\
		\hline
		\hline
	\end{tabular}
\caption{Computational complexity of the first-order oracle in \ref{slep}. Here, $FO$ denotes the computational cost of computing the gradient of $f$.}
\label{ta:comput_slepgra}
\end{table}

In the remainder of this subsection, we present the closed-form expression of $\nabla^2 h$ under the assumption that the objective function $f$ is second-order differentiable.

\begin{assumption}[Second-order differentiability of $f$]\label{as:s}
   The Hessian operator $\nabla^2 f(X)$ exists for all $X\in\mathbb{R}^{n\times p}$.
\end{assumption}
The following proposition provides the closed-form expression of $\nabla^2 h$.
\begin{proposition}
Under Assumption \ref{as:s}, the action of the Hessian operator $\nabla^2 h(X)$ at $X\in\Rnp$ along direction $D\in\mathbb{R}^{n\times p}$ is given by
	\begin{equation}\label{eq:hessh}
	    \begin{aligned}
		&\nabla^2 h(X)[D]=\mathcal{J}_G(X)[D]\left(\frac{3}{2}I_p-\frac{1}{2}X^\top MX\right)-G(X)\Psi\left(X^\top MD\right)
		-MD\Psi\left(X^\top G(X)\right)\\
		&-MX\Psi\left(D^\top G(X)\right)
		-MX\Psi\left(X^\top \nabla G(X)[D]\right)+\beta MD\left(X^\top MX-I_p\right)+2\beta MX\Psi\left(X^\top MD\right),\\
	\end{aligned}
	\end{equation}
 where $\mathcal{J}_G(X)$ denotes the Jacobian operator of the mapping $G$ with respect to $X$, defined as
 $\mathcal{J}_G(X)[D]:=\nabla^2 f\left(X\left(\frac{3}{2}I_p-\frac{1}{2}X^\top MX\right)\right)\left[D\left(\frac{3}{2}I_p-\frac{1}{2}X^\top MX\right)-X\Psi\left(X^\top MD\right)\right].$
\end{proposition}
\subsection{Equivalence between \ref{slep} and \ref{ocp}}
\label{relation}
In this subsection, we establish the equivalence between the penalty model \ref{slep} and the original problem \ref{ocp} in the neighborhood of the feasible region $\Smnp$. Specifically, we demonstrate that \ref{slep} and \ref{ocp} share the same first- and second-order stationary points, local minimizers, and satisfy the Łojasiewicz gradient inequality in this neighborhood. Furthermore, we demonstrate that the stationarity violation in \ref{ocp} is bounded by that in \ref{slep}. 

This subsection begins with several preliminary propositions.
The following proposition provides an estimate of the distance between a point in the neighborhood of the manifold and its images under the mappings $\mathcal{A}$ and $\mathcal{R}$.
\begin{proposition}\label{pr:dist}
        For any feasible point $X\in\mathcal{S}_M(n,p)$ and any point $Y$ in its neighborhood $\mathbb{B}(X,\rho_X)$, the following inequalities hold: $ \|Y-\mathcal{R}(Y)\|\le\frac{3}{2}\|Y\|_F\|Y^\top MY-I_p\|_2,  \|Y-\mathcal{A}(Y)\|_F\le\frac{1}{2}\|Y\|_F\|Y^\top MY-I_p\|_2,$ and $
         \left\|\mathcal{A}(Y)-\mathcal{R}(Y)\right\|_F\le
   \frac{3}{4}\|Y\|_F\|Y^\top MY-I_p\|_2^2.$
\end{proposition}
\begin{proof}
From the definition of $\rho_X$, the eigenvalues of $Y^\top MY$ lie within the interval
\(\bigl[\tfrac{5}{6},\tfrac{7}{6}\bigr]\),
which guarantees that both $(Y^\top MY)^{1/2}$ and $(Y^\top MY)^{-1/2}$ are well defined and uniformly bounded. 
These bounds can be verified by direct calculation, as follows:
\[  \begin{aligned}
     \|Y-\mathcal{R}(Y)\|_F&=\left\|Y-Y\left(Y^\top MY\right)^{-\frac{1}{2}}\right\|_F=\left\|Y\left(Y^\top MY\right)^{-\frac{1}{2}}\left(\left(Y^\top MY\right)^{\frac{1}{2}}-I_p\right)\right\|_2\\
        &=\left\|Y\left(Y^\top MY\right)^{-\frac{1}{2}}\left(\left(Y^\top MY\right)^{\frac{1}{2}}+I_p\right)^{-1}\left(Y^\top MY-I_p\right)\right\|_2\\
        &\le\left\|Y\right\|_F\left\|\left(Y^\top MY\right)^{-\frac{1}{2}}\left(\left(Y^\top MY\right)^{\frac{1}{2}}+I_p\right)^{-1}\right\|_2\|Y^\top MY-I_p\|_2\overset{(i)}{\le}\frac{3}{2}\|Y\|_F\|Y^\top MY-I_p\|_2,
\end{aligned}\]
 where the inequality $(i)$ follows from
\[\begin{aligned}
    \left\|\left(Y^\top MY\right)^{-\frac{1}{2}}\left(\left(Y^\top MY\right)^{\frac{1}{2}}+I_p\right)^{-1}\right\|_2\le\frac{1}{\lambda_{\min}^{\frac{1}{2}}(X^\top MX)(\lambda_{\min}^{\frac{1}{2}}(X^\top MX)+1)}\le\frac{3}{2}.
\end{aligned}\]
Similiarly, for the bound on $ \|Y-\mathcal{A}(Y)\|_F$, we obtain
    \[\begin{aligned}
        \|Y-\mathcal{A}(Y)\|_F&=\left\|Y-Y\left(\frac{3}{2}I_p-\frac{1}{2}Y^\top MY\right)\right\|_F\le\frac{1}{2}\|Y\|_F\|Y^\top MY-I_p\|_2.\\
\end{aligned}\]
Finally, for the bound on $\|\mathcal{A}(Y)-\mathcal{R}(X)\|_F$, we have
\[\begin{aligned}
     \left\|\mathcal{A}(Y)-\mathcal{R}(X)\right\|_F
    &=\left\|Y(Y^\top MY)^{-\frac{1}{2}}\left(\frac{3}{2}(Y^\top MY)^{\frac{1}{2}}-\frac{1}{2}(Y^\top MY)^{\frac{3}{2}}-I_p\right)\right\|_F\\
    &=\frac{1}{2}\left\|Y(Y^\top MY)^{-\frac{1}{2}}\left((Y^\top MY)^{\frac{1}{2}}+2I_p\right)\left((Y^\top MY)^{\frac{1}{2}}-I_p\right)^2\right\|_F\\
    &=\frac{1}{2}\left\|Y(Y^\top MY)^{-\frac{1}{2}}\left((Y^\top MY)^{\frac{1}{2}}+2I_p\right)\left(Y^\top MY\right)^{\frac{1}{2}}+I_p)^{-2}\left(Y^\top MY-I_p\right)^2\right\|_F\\
       &=\frac{1}{2}\left\|Y\right\|_F\left\|\left(I_p+2(Y^\top MY)^{-\frac{1}{2}}\right)\left(\left(Y^\top MY\right)^{\frac{1}{2}}+I_p\right)^{-2}\right\|_2\left\|Y^\top MY-I_p\right\|_2^2\\
    &\overset{(ii)}{\le}\frac{3}{4}\|Y\|_F\|Y^\top MY-I_p\|_2^2,
\end{aligned}\]   
Where inequality $(ii)$ follows from
\[\begin{aligned}
    &\left\|\left(I_p+2(Y^\top MY)^{-\frac{1}{2}}\right)\left(\left(Y^\top MY\right)^{\frac{1}{2}}+I_p\right)^{-2}\right\|_2
\le\frac{1+2\lambda_{\min}^{-\frac{1}{2}}(X^\top MX)}{\lambda_{\min}^{\frac{1}{2}}(X^\top MX)(\lambda_{\min}^{\frac{1}{2}}(X^\top MX)+1)^2}\le\frac{3}{2}.
\end{aligned}\]
The proof is complete.
\end{proof}
Based on this proposition and the definition of $\rho_X$, we immediately obtain the following corollary.
\begin{proposition}\label{pr:ar}
     For any feasible point $X\in\mathcal{S}_M(n,p)$ and any point $Y$ in its neighborhood $\mathbb{B}(X,\rho_X)$, 
the images of $Y$ under the mappings $\mathcal{A}$ and $\mathcal{R}$ remain within the unit neighborhood of $X$; that is,  $\mathcal{A}(Y),\mathcal{R}(Y)\in\mathbb{B}(X,1)$.
\end{proposition}

The following proposition characterizes a key property of the gradient $\nabla g(X)$.
\begin{proposition}
For any $Y\in\mathbb{R}^{n\times p}$, the following identity holds:
\begin{equation}\label{eq:gradh_inner}
    \left\langle\nabla h(Y),\,Y\left(Y^\top MY-I_p\right)\right\rangle
    = \left\langle\beta X^\top MX-\frac{3}{2}\langle \Psi\!\left(Y^\top G(Y)\right),\,\left(Y^\top MY-I_p\right)^2\right\rangle.
\end{equation}
\end{proposition}
\begin{proof}
     According to the expression for $\nabla h$ in \eqref{eq:grad_h}, we have
\[
\begin{aligned}
&\left\langle\nabla h\left(Y\right),Y\left(Y^{\top}MY-I_p\right)\right\rangle \\
= &\left\langle G\left(Y\right)\left(\frac{3}{2}I_p-\frac{1}{2}Y^{\top} M Y\right)-M Y\Psi\left(Y^{\top} G\left(Y\right)\right)+\beta MY(Y^{\top} MY),Y\left(Y^{\top} M Y-I_p\right)\right\rangle \\
= &\tr{} \left( \left(Y^{\top}MY-I_p\right)Y^{\top}G(Y)\left(\frac{3}{2}I_p-\frac{1}{2}Y^{\top} MY\right) \right) -\tr \left(\left(Y^{\top}MY-I_p\right) Y^{\top} MY\Psi\left(Y^{\top} G(Y)\right)\right) \\
&\quad +\beta\tr\left(\left(Y^{\top}MY-I_p\right) Y^{\top} MY\left(Y^{\top} MY-I_p\right)\right) \\
\overset{(i)}{=} &\tr\left(\Psi\left(Y^{\top} G\left(Y\right)\right)\left(\frac{3}{2}I_p-\frac{1}{2}Y^{\top} MY\right)\left(Y^{\top} MY-I_p\right)\right) -\tr\left(\Psi\left(Y^{\top} G\left(Y\right)\right)Y^{\top} MY\left(Y^{\top} MY-I_p\right)\right) \\
&\quad +\beta\tr\left(Y^{\top} MY\left(Y^{\top} MY-I_p\right)^2\right) \\
= & -\frac{3}{2}\tr\left(\Psi\left(Y^{\top} G\left(Y\right)\right)\left(I_p-Y^{\top} MY\right)^2\right) +\beta\tr\left(Y^{\top} MY\left(Y^{\top} MY-I_p\right)^2\right) \\
= &\tr\left(\left[\beta Y^{\top} MY-\frac{3}{2}\Psi\left(Y^{\top} G\left(Y\right)\right)\right]\left(Y^{\top} M Y-I_p\right)^2\right) = \left\langle\beta Y^{\top} MY-\frac{3}{2}\Psi\left(Y^{\top} G(Y)\right),\left(Y^{\top} MY-I_p\right)^2\right\rangle
\end{aligned}
\]
where the equality $(i)$ follows from the identity $\tr(AB)=\tr(\Psi(A)B)$, which holds for any matrix $A\in\mathbb{R}^{p\times p}$ and symmetric matrix $B\in\mathbb{S}^p$. 
The proof is complete.
\end{proof}

\subsubsection{Equivalence on first-order stationary points}\label{relation1}
In this section, we discuss the relationship between the first-order stationary points of \ref{slep} and those of \ref{ocp}. We first present the following proposition to illustrate that both formulations share the same stationary structure over the manifold $\Smnp$.
 \begin{proposition}\label{pr:equa_fo_smnp}
Under Assumption \ref{as:blanket}, the penalty model \ref{ocp} and the original problem \ref{slep} share the same first-order stationary points on the manifold $\mathcal{S}_M(n,p)$.
 \end{proposition}
 \begin{proof}
  For any $X\in\Smnp$, it follows from the feasibility of $X$ (i.e., $X\tp MX = I_p$) and the expression for $\nabla h(X)$ in \eqref{eq:grad_h} that 
      \begin{equation}\label{eq:relat_grad_smnp}
          \nabla h(X)=\bar{M}(\grad  f(X)-\mathcal{P}_1(X)\Psi(\mathcal{P}_2(X)^\top \nabla f(X))).
      \end{equation}
   If $X \in \Smnp$ is a first-order stationary point of the original problem \ref{ocp}, then by Definition \ref{def:o_ocp}, it holds that $0 = \grad  f(X)$. Moreover, using
    \begin{equation}\label{eq:p2}
        \mathcal{P}_2(X)^\top \nabla f(X)=\mathcal{P}_2(X)^\top\bar{M}(\grad  f(X)+\mathcal{P}_1(X)\Psi(\mathcal{P}_1(X)^\top \nabla f(X)))=\mathcal{P}_2(X)^\top\bar{M}\grad  f(X)=0,
    \end{equation}
    we obtain $\nabla h(X)=0$. Hence, $X$ is also a first-order stationary point of the penalty model \ref{slep}. 

   Conversely, if $X \in \Smnp$ is a first-order stationary point of the penalty model \ref{slep}, then by Definition \ref{def:o_slep}, we have $\nabla h(X)= 0$. Substituting into \eqref{eq:relat_grad_smnp} gives $\grad  f(X)-\mathcal{P}_1(X)\Psi(\mathcal{P}_2(X)^\top \nabla f(X))=0$. Since $\grad  f(X)\in\mathcal{T}_X$ and $\mathcal{P}_1(X)\Psi(\mathcal{P}_2(X)^\top \nabla f(X))\in\mathcal{N}_X$, it follows that $\grad  f(X)=0$ and $\mathcal{P}_1(X)\Psi(\mathcal{P}_2(X)^\top \nabla f(X))=0$. Therefore, by Definition \ref{def:o_ocp}, $X$ is also a first-order stationary point of \ref{ocp}. This completes the proof.  
 \end{proof} 
Furthermore, the following theorem shows that when the penalty parameter $\beta$ is sufficiently large, any infeasible first-order stationary point of \ref{slep} 
is bounded away from the manifold $\Smnp$.

 \begin{theorem}\label{th:kappaxmx}
	Under Assumption \ref{as:blanket}, for any $X\in\mathcal{S}_M(n,p)$, suppose $\beta>\beta_X$. Then, the infeasible first-order stationary point $X^*\in\mathbb{B}(X,1)\backslash\Smnp$ of \ref{slep} satisfies $\lambda_{\min}({X^*}^\top M X^*)\le \frac{1}{6}.$
 \end{theorem}
\begin{proof}
We prove this theorem by contradiction. Assume that $\lambda_{\min}({X^*}^\top M X^*)>\frac{1}{6}$. It follows from Equation \eqref{eq:gradh_inner} that 
\[\begin{aligned}
& \left\langle\nabla h\left(X^*\right), X^*\left({X^*}^\top M X^*-I_p\right) \right\rangle \\
= & \beta\left\langle M X^*\left({X^*}^\top M X^*-I_p\right),X^*\left({X^*}^\top M X^*-I_p\right) \right\rangle - \frac{3}{2} \left\langle \Psi\left({X^*}^\top G\left(X^*\right)\right),\left({X^*}^\top M X^*-I_p\right)^2\right\rangle\\
\ge & \beta\left\langle M X^*\left({X^*}^\top M X^*-I_p\right),X^*\left({X^*}^\top M X^*-I_p\right) \right\rangle - \frac{3}{2}\left| \left\langle \Psi\left({X^*}^\top G\left(X^*\right)\right),\left({X^*}^\top M X^*-I_p\right)^2\right\rangle\right| \\
\ge & \beta\left\langle {X^*}^\top MX^*,\left({X^*}^\top M X^*-I_p\right)^2\right\rangle - \frac{3}{2} \left\langle \|X^*\|_F\|G\left(X^*\right)\|_FI_p,\left({X^*}^\top M X^*-I_p\right)^2\right\rangle \\
\ge & \left\langle \beta{X^*}^\top M X^*-H_{X,G}C_XI_p,\left({X^*}^\top MX^*-I_p\right)^2 \right\rangle\\
\ge & \left\langle(\beta \lambda_{\min}({X^*}^\top M X^*)-H_{X,G}C_XI_p)I_p,\left({X^*}^\top MX^*-I_p\right)^2 \right\rangle\\
\ge & \left\langle\left(\frac{\beta}{6}-H_{X,G}C_XI_p\right)I_p,\left({X^*}^\top MX^*-I_p\right)^2 \right\rangle\overset{(i)}{>}0,\\
\end{aligned}\]
where the inequality $(i)$ follows from the assumption that $\beta>\beta_X\ge 6H_{X,G}C_X$.
This contradicts the fact that $X^*$ is a first-order stationary point of \ref{slep}. Hence, it must hold that $\lambda_{\min}({X^*}^\top M X^*)\le \frac{1}{6}$.
 The proof is complete.
 \end{proof}

The following result, which follows directly from the theorem above, establishes the equivalence between the first-order stationary points 
of the penalty model \ref{slep} and those of the original problem \ref{ocp} 
in a neighborhood of an arbitrary feasible point.
\begin{corollary}\label{co:fo}
    Under Assumption \ref{as:blanket}, suppose $X\in\mathcal{S}_M(n,p)$ and  $\beta>\beta_{X}$. Then, \ref{slep} and \ref{ocp} share the same first-order stationary points in $\mathbb{B}(X,\rho_X)$.
\end{corollary}
\begin{proof}
   From Theorem \ref{th:kappaxmx}, any infeasible first-order stationary point $X^*\in\mathbb{B}(X,1)\setminus\mathcal{S}_M(n,p)$ of the penalty model \ref{slep} satisfies $\lambda_{\min}({X^*}^\top M X^*) \le \frac{1}{6}.$ 
According to the definition of $\rho_X$, every matrix $Y\in\mathbb{B}(X,\rho_X)$ satisfies $\frac{5}{6} \le \lambda_{\min}(Y^\top M Y) \le \lambda_{\max}(Y^\top M Y) \le \frac{7}{6}.$ 
Therefore, an infeasible stationary point $X^*$ cannot lie in $\mathbb{B}(X,\rho_X)$. Consequently, all first-order stationary points of \ref{slep} contained in $\mathbb{B}(X,\rho_X)$ must be feasible, and hence coincide with those of the original problem \ref{ocp} by Proposition \ref{pr:equa_fo_smnp}. This completes the proof.
\end{proof}

\subsubsection{Equivalence on second-order stationary points}\label{relation2}
 In this section, we establish the equivalence between the second-order stationary points of \ref{slep} and \ref{ocp} in the neighborhood of an arbitrary feasible point. In particular, the following proposition shows that any infeasible first-order stationary point of \ref{slep} becomes a saddle point in \ref{slep} when the penalty parameter $\beta$ is sufficiently large.

\begin{proposition}\label{pr:info}
	Under Assumptions \ref{as:blanket} and \ref{as:s}, for any feasible point $X\in\mathcal{S}_M(n,p)$ and penalty parameter $\beta\ge\beta_X$, every infeasible first-order stationary point $X^*\in\mathbb{B}(X,1)$ of the penalty model \ref{slep} is a $(\beta/4C_X^2)$-strict saddle point.
\end{proposition}
\begin{proof}
	Let $v$ be a unit eigenvector of ${X^*}^\top M X^*$ corresponding to its smallest eigenvalue, that is, $({X^*}^\top MX^*)v=\lambda_{\min}({X^*}^\top MX^*)v$.
Define $D:=X^*vv^\top$. From Taylor’s expansion of the function $c(X):=\|X^\top MX-I_p\|_F^2$ along the direction $D$, we obtain
\[
\begin{aligned}
c(X^*+tD)
&=\|(X^*+tD)^\top M(X^*+tD)-I_p\|_F^2\\
&=\|{X^*}^\top MX^*-I_p+2\lambda_{\min}({X^*}^\top MX^*)t vv^\top+\lambda_{\min}({X^*}^\top MX^*)t^2 vv^\top\|_F^2\\
&=\|{X^*}^\top MX^*-I_p\|_F^2
+4\lambda_{\min}({X^*}^\top MX^*)(\lambda_{\min}({X^*}^\top MX^*)-1)t\\
&\quad+4\lambda_{\min}({X^*}^\top MX^*)^2t^2
+2\lambda_{\min}({X^*}^\top MX^*)(\lambda_{\min}({X^*}^\top MX^*)-1)t^2+O(t^3)\\
&=4c(X^*)+4\lambda_{\min}({X^*}^\top MX^*)(\lambda_{\min}({X^*}^\top MX^*)-1)t\\
&\quad+4\lambda_{\min}({X^*}^\top MX^*)(3\lambda_{\min}({X^*}^\top MX^*)-1)\frac{t^2}{2}+O(t^3).
\end{aligned}
\]
Therefore,
\[
\left\langle D,\nabla^2 c(X^*)[D]\right\rangle
=4\lambda_{\min}({X^*}^\top MX^*)(3\lambda_{\min}({X^*}^\top MX^*)-1).
\]
Since $h(X)=g(X)+\frac{\beta}{4}c(X)$, we obtain
\[
\begin{aligned}
\left\langle D,\nabla^2 h(X^*)[D]\right\rangle
&=\left\langle D,\nabla^2 g(X^*)[D]\right\rangle
+\beta\left\langle D,\nabla^2 c(X^*)[D]\right\rangle\\
&\le \|\nabla^2 g(X^*)\|_F\|D\|_F^2
+\beta\lambda_{\min}({X^*}^\top MX^*)(3\lambda_{\min}({X^*}^\top MX^*)-1)\\
&\overset{(i)}{\le}(L_{X,\nabla g}-\frac{\beta}{2C_X^2})\|D\|_F^2
\le -\frac{\beta}{4C_X^2}\|D\|_F^2,
\end{aligned}
\]
Where the inequality $(i)$ follows $\lambda_{\min}({X^*}^\top MX^*)\le\frac{1}{6}$. 
Therefore, $X^*$ is a $(\beta/4C_X^2)$-strict saddle point of \ref{slep}. The proof is complete.
\end{proof}

The following proposition depicts essential properties of the Hessian $\nabla^2 h(X)$ at a feasible first-order stationary point. 
\begin{proposition}
    \label{pr:hessh}
	Under Assumptions \ref{as:blanket} and \ref{as:s}, for any feasible point $X\in\mathcal{S}_M(n,p)$, suppose $X^*\in\mathbb{B}(X,1)\cap\Smnp$ is a first-order stationary point of \ref{slep}. 
	Then, for any $T\in\mathcal{T}_{X^*}$ and $N\in\mathcal{N}_{X^*}$, the following hold:
 \begin{align}
     & \left\langle T,\nabla^2  h(X^*)[T]\right\rangle=\left\langle T,\operatorname{hess} f(X^*)[T]\right\rangle,\\
	    &\left\langle N,\nabla^2  h(X^*)[T]\right\rangle=0.
 \end{align}
	   	Moreover, when $\beta\ge\beta_X$, for any $0\neq N\in\mathcal{N}_{X^*}$ it follows that
	\begin{equation}\label{ww}
	    \left\langle N,\nabla^2 h(X^*)[N] \right\rangle>0.
	\end{equation}
\end{proposition}

\begin{proof}
Since $X^*$ is feasible, from the expression \eqref{eq:hessh} of the Hessian operator $\nabla^2 h$, for every $T\in\mathcal{T}_{X^*}$ and $N=X^*S\in\mathcal{N}_{X^*}$ we obtain
\[\begin{aligned}
    \nabla^2 h(X^*)[T]&=\nabla^2 f(X^*)[T]-MX^*\Psi({X^*}^\top \nabla^2 f(X)[T])-MT\Psi({ X^*}^\top\nabla f(X^*))-MX^*\Psi(T^\top\nabla f(X^*))\\
 \nabla^2 h(X^*)[N]&=-\nabla f(X^*)S-MN\Psi({X^*}^\top \nabla f(X^*))-MX^*\Psi(N^\top \nabla f(X^*))+2\beta MN.
\end{aligned}
\]
By direct calculation, we have
 \[\begin{aligned}
 &\left\langle T,\nabla^2 h(X^*)[T]\right\rangle\\
  =&\left\langle T,\nabla^2 f(X^*)[T]\right\rangle-\left\langle T,MT\Psi({X^*}^\top \nabla f(X^*))\right\rangle\\
  &-\left\langle T,MX^*\Psi({X^*}^\top \nabla^2 f(X)[T])\right\rangle-\left\langle T,MX^*\Psi(T^\top\nabla f(X^*))\right\rangle\\
 \overset{(i)}{=}&\left\langle T,\nabla^2 f(X^*)[T]\right\rangle-\left\langle T,MT\Psi({X^*}^\top \nabla f(X^*))\right\rangle\\
  \overset{(ii)}{=}&\left\langle T,\nabla^2 f(X^*)[T]\right\rangle-\left\langle T,MT\Psi(\mathcal{P}_1(X^*)^\top \nabla f(X^*))\right\rangle=\left\langle T,\operatorname{hess} f(X^*)[T]\right\rangle_{\bar{M}}.\\
		&\left\langle N, \nabla^2 h(X^*)[T]\right\rangle\\
		=&\left\langle N,\nabla^2 f(X^*)[T]-MX^*\Psi({X^*}^\top \nabla^2 f(X^*)[T])\right\rangle-\left\langle N, MT\Psi({X^*}^\top\nabla f(X^*))\right\rangle
	      \\&-\left\langle N, MX^*\Psi(T^\top \nabla f(X^*))\right\rangle\\
		=&\tr(S{X^*}^\top\nabla^2 f(X^*)[T])-\tr(S\Psi({X^*}^\top \nabla^2 f(X^*)[T]))-\tr(S{X^*}^\top M T\Psi({X^*}^\top \nabla f(X^*)))\\
&-\tr(S{X^*}^\top MX^*\Psi(T^\top \nabla f(X^*)))\\
		=&\tr(ST^\top M X^*\Psi({X^*}^\top \nabla f(X^*)))-\tr(S\Psi(T^\top \nabla f(X^*)))\\
		\overset{(iii)}{=}&\tr(ST^\top \nabla f(X^*))-\tr(ST^\top \nabla f(X^*))=0,\\
	&\left\langle N,\nabla^2 h(X^*)[N]\right\rangle\\ =&-\left\langle N,\nabla f(X^*)S\right\rangle-
	\left\langle N,	MN\Psi({X^*}^\top \nabla f(X^*)) \right\rangle -\left\langle N, MX^*\Psi(N^\top\nabla f(X^*))\right\rangle+2\beta\left\langle N,MN \right\rangle\\
	=&2\beta||S||_F^2-3\tr(S^2{X^*}^\top\nabla f(X^*))
	\ge(2\beta-3H_{X,\nabla f}C_X)\|S\|_F^2>0,
\end{aligned}\]
where the equation $(i)$ holds since for every $S\in\mathbb{S}^p$, $\langle T,MX^*S\rangle=\langle T,\bar{M}\mathcal{P}(X^*)S\rangle=\langle T,\mathcal{P}(X^*)S\rangle_{\bar{M}}=0$, the equation $(ii)$ follows from \eqref{eq:p2}, and the equation $(iii)$ holds
because $\nabla f(X^*)=MX^*\Psi({X^*}^\top \nabla f(X^*))$. Thus, when $\beta\ge\beta_X$, inequality \eqref{ww} holds, and the proof is complete.
\end{proof}
In the following theorem, based on Propositions \ref{pr:info} and \ref{pr:hessh}, we establish the equivalence between the second-order stationary points of \ref{ocp} and \ref{slep} within $\mathbb{B}(X,1), X\in\Smnp$.
\begin{theorem}\label{th:sopr_b1}
    Under Assumptions \ref{as:blanket} and \ref{as:s}, for any feasible point $X\in\mathcal{S}_M(n,p)$, suppose that the penalty parameter $\beta\ge\beta_X$; then \ref{slep} and \ref{ocp} share the same second-order stationary points in $\mathbb{B}(X,1)$.
\end{theorem}
\begin{proof}
According to Proposition \ref{pr:info}, when $\beta\ge\beta_X$, if $X^*\in\mathbb{B}(X,\rho_X)$ is a second-order stationary point of \ref{slep}, then $X^*\in\Smnp$. Hence, $X^*$ is also a first-order stationary point of \ref{ocp}. From Proposition \ref{pr:hessh}, we have  $\left\langle T,\operatorname{hess} f(X^*)[T]\right\rangle_{\bar{M}}=\left\langle T,\nabla^2  h(X^*)[T]\right\rangle,\forall T\in\mathcal{T}_X\ge0$.  Therefore, $X$ is a second-order stationary point of \ref{ocp}.
 Conversely, suppose $X^*\in\mathbb{B}(X,1)$ is a second-order stationary point of \ref{ocp}.
 For every $D\in\mathbb{R}^{n\times p}$, there exists $T\in\mathcal{T}_{X^*}$ and $N\in\mathcal{N}_{X^*}$ such that $D=T+N$.
From Proposition \ref{pr:hessh}, we have
 \[
 \begin{aligned}
     \left\langle D, \nabla^2 h(X^*)[D] \right\rangle
     &=\left\langle T,\nabla^2 h(X^*)[T] \right\rangle+2\left\langle N,\nabla^2 h(X^*)[T] \right\rangle+\left\langle N,\nabla^2 h(X^*)[N] \right\rangle\\
     &=\left\langle T,\operatorname{hess} f(X)[T] \right\rangle_{\bar{M}}+\left\langle N,\nabla^2 h(X^*)[N] \right\rangle
     \ge0
 \end{aligned}
 \]
 Thus, $X^*$ is a second-order stationary point of \ref{slep}. The proof is complete.

\end{proof}
\subsubsection{Equivalence on stationarity violation}\label{restation}
Note that \ref{ocp} is a Riemannian optimization problem, and its stationarity violation is characterized by $\|\grad f(X)\|_F$. However, because \ref{slep} is an unconstrained optimization problem, its stationarity violation is measured by $\|\nabla h(X)\|_F$. Therefore, in this section, we establish the relationship between the stationarity violations of \ref{slep} and \ref{ocp}, which is crucial for designing termination criteria and analyzing iteration complexity for algorithms developed based on \ref{slep}.

We begin our analysis with the following proposition. 
\begin{proposition}\label{pro:sta_slep}
     Under Assumption \ref{as:blanket}, for any feasible point $X\in\mathcal{S}_M(n,p)$ and any point $Y$ in its neighborhood $\mathbb{B}(X,\rho_X)$, suppose $\beta\ge\beta_{X}$. Then, it holds that
     \[\|\nabla h(Y)\|_F\ge \frac{1}{2\kappa^{\frac{1}{2}}(\bar{M})}\left\|
    \nabla g(Y)\right\|_F+\frac{\lambda_{\min}^{\frac{1}{2}}(\bar{M})\beta}{4}\left\|Y^\top MY-I_p\right\|_F.\]
\end{proposition}
\begin{proof}
    From the expression of $\nabla h(Y)$, we have
 \[\begin{aligned}
        &\|\nabla h(Y)\|^2_F\\
        \ge{}&\lambda_{\min}(\bar{M})\tr(\nabla h(Y)^\top {\bar{M}}^{-1}\nabla h(Y))\\
    \ge{}&  \frac{1}{\kappa(\bar{M})}\left\|\nabla g(Y)\right\|^2_F+2\lambda_{\min}(\bar{M})\beta\left\langle\nabla g(Y),Y\left(Y^\top MY-I_p\right)\right\rangle\\
    &+\lambda_{\min}(\bar{M})\beta^2\tr\left(\left(Y^\top MY-I_p\right)Y^\top MY\left(Y^\top MY-I_p\right)\right)\\
    \overset{(i)}{\ge}{}& \frac{1}{\kappa(\bar{M})}\left\|\nabla g(Y)\right\|^2_F-3\lambda_{\min}(\bar{M})\beta\left\langle \Psi\left({\mathcal{P}_1(Y)}^\top G\left(Y\right)\right),\left({Y}^\top M Y-I_p\right)^2\right\rangle 
   +\frac{5\lambda_{\min}(\bar{M})\beta^2}{6}\left\|Y^\top MY-I_p\right\|^2_F\\
    \ge{}&  \frac{1}{\kappa(\bar{M})}\left\|\nabla g(Y)\right\|^2_F-3H_{X,G}\lambda_{\min}(\bar{M})C_X\beta\left\|Y^\top MY-I_p\right\|_F^2+\frac{5\lambda_{\min}(\bar{M})\beta^2}{6}\left\|Y^\top MY-I_p\right\|^2_F\\
    \overset{(ii)}{\ge}{}&   \frac{1}{\kappa(\bar{M})}\left\|
    \nabla g(Y)\right\|^2_F+\frac{\lambda_{\min}(\bar{M})\beta^2}{4}\left\|Y^\top MY-I_p\right\|^2_F.
    \end{aligned}
    \]
    Here, the inequality $(i)$ uses equation \eqref{eq:gradh_inner}, and $(ii)$ holds because $\beta\ge\beta_X\ge 6H_{X,G}C_X$. Therefore, 
\[\|\nabla h(Y)\|_F\ge \frac{1}{2\kappa^{\frac{1}{2}}(\bar{M})}\left\|
    \nabla g(Y)\right\|_F+\frac{\lambda_{\min}^{\frac{1}{2}}(\bar{M})\beta}{4}\left\|Y^\top MY-I_p\right\|_F.\]
  Therefore, the proof is complete.
\end{proof}

The following theorem states that the stationary violation of \ref{ocp} can be bounded by that of \ref{slep} in $\mathbb{B}(X,\rho_X)$ for any $X\in\Smnp$.
\begin{theorem}\label{th:star_bx}
    Under Assumption \ref{as:blanket}, for any feasible point $X\in\mathcal{S}_M(n,p)$ and any point $Y$ in its neighborhood $\mathbb{B}(X,\rho_X)$, suppose $\beta\ge\beta_{X}$. Then it holds that
    \[\|\nabla h(Y)\|_F\ge \frac{\lambda_{\min}(\bar{M})}{2\kappa^{\frac{1}{2}}(\bar{M})(1+\lambda_{\max}(\bar{M})C_X^2)}\|\grad  f(\mathcal{R}(X))\|_F+\frac{\lambda_{\min}^{\frac{1}{2}}
 (\bar{M})\beta}{8}\left\|Y^\top MY-I_p\right\|_F.\].
\end{theorem}
\begin{proof}
Because $\mathcal{R}(Y)\in\Smnp$, we compare the expressions of $\nabla g(\mathcal{R}(Y))$ and $\grad  f(\mathcal{R}(Y))$. Thus, we have
\[\begin{aligned}
    \nabla g(\mathcal{R}(Y))&=\bar{M}\grad  f(\mathcal{R}(Y))-M\mathcal{R}(Y)\Psi(\mathcal{P}_2(\mathcal{R}(Y))^\top\nabla f(\mathcal{R}(Y)))\\
    &=\bar{M}\grad  f(\mathcal{R}(Y))-M\mathcal{R}(Y)\Psi(\mathcal{P}_2(\mathcal{R}(Y))^\top \nabla g(\mathcal{R}(Y))).
\end{aligned}\]
Then we obtain the following estimation for $\|\grad  f(\mathcal{R}(Y))\|_F$:
\[\begin{aligned}
    &\lambda_{\min}(\bar(M))\|\grad f(X)\|_F\le \| \bar{M}\grad f(\mathcal{R}(Y))\|_F\\
    \le{}&\|\nabla g(\mathcal{R}(Y))-\nabla g(Y)+\nabla g(Y)\|_F+\|M\mathcal{R}(Y)\Psi(\mathcal{P}_2(\mathcal{R}(Y))(\nabla g(\mathcal{R}(Y)-\nabla g(Y)+\nabla g(Y)))\|_F\\
    \le{}&(1+\lambda_{\max}(M)\|\mathcal{R}(Y)\|_F^2)(\|\nabla g(Y)\|_F+\|\nabla g(\mathcal{R}(Y))-\nabla g(Y)\|_F)\\
    \overset{(i)}{\le}{}&(1+\lambda_{\max}(M)C_X^2)(\|\nabla g(Y)\|_F+H_{X,\nabla g}\|\mathcal{R}(Y)-Y\|_F)\\
     \overset{(ii)}{\le}{}&(1+\lambda_{\max}(M)C_X^2)(\|\nabla g(Y)\|_F+\frac{3H_{X,\nabla g}}{2}\|Y\|_F\|Y^\top MY-I_p\|_F)\\
\le{}&(1+\lambda_{\max}(M)C_X^2)(\|\nabla g(Y)\|_F+\frac{3H_{X,\nabla g}C_X}{2}\left\|Y^\top MY-I_p\right\|_F),
\end{aligned}\]
where the inequality $(i)$ follows from Proposition \ref{pr:ar} and the inequality $(ii)$ follows from Proposition \ref{pr:dist}. When $\beta\ge\beta_X\ge\frac{6\lambda_{\min}^{\frac{1}{2}}(\bar{M})H_{X,\nabla g}C_X}{\kappa^{\frac{1}{2}}(\bar{M})(1+\lambda_{\max}C_X^2)}$, it holds that 
\[\begin{aligned}
    &\|\nabla h(Y)\|_F\\
    \ge&\frac{\lambda_{\min}(\bar{M})}{2\kappa^{\frac{1}{2}}(\bar{M})(1+\lambda_{\max}(\bar{M})C_X^2)}\left(\|\grad  f(X)\|_F-\frac{3H_{X,\nabla g}C_X}{2}\left\|Y^\top MY-I_p\right\|_F\right)
    \\&+\frac{\lambda_{\min}^{\frac{1}{2}}(\bar{M})\beta}{4}\left\|Y^\top MY-I_p\right\|_F \\
    \ge&\frac{\lambda_{\min}(\bar{M})}{2\kappa^{\frac{1}{2}}(\bar{M})(1+\lambda_{\max}(\bar{M})C_X^2)}\|\grad  f(X)\|_F+\frac{\lambda_{\min}^{\frac{1}{2}}(\bar{M})\beta}{8}\left\|Y^\top MY-I_p\right\|_F.\\
\end{aligned}\]
This completes the proof. 
\end{proof}

\subsubsection{Equivalence on local minimizers}\label{relocal}
In this section, we analyze the relationship between the local minimizers of \ref{slep} and \ref{ocp}. 
The following proposition characterizes the changes in the function value of $h$ after performing a generalized orthogonalization of $Y\in\mathbb{B}(X,\frac{\rho}{2})$ for any $X\in\mathcal{S}_M(n,p)$.
\begin{proposition}\label{pr:varh}
    Under Assumption \ref{as:blanket}, for any feasible point $X\in\mathcal{S}_M(n,p)$ and any point $Y$ in the neighborhood $\mathbb{B}(X,\rho_X)$,
    suppose $\beta\ge \beta_X$, then  we have
  \[h(\mathcal{R}(Y))-h(Y)\le-\frac{\beta}{8}\|Y^\top MY-I_p\|_F^2.\]
\end{proposition}
\begin{proof}
From the expression of $h$ in \ref{slep}, we obtain
    \[\begin{aligned}
        &h(Y)-h(\mathcal{R}(Y))= f\left(\mathcal{A}(Y)\right)-f(\mathcal{R}(Y))+\frac{\beta}{4}\left \|Y^\top MY-I_p\right\|_F^2\\
        \ge&-H_{X,\nabla f}\|\mathcal{A}(Y)-\mathcal{R}(Y)\|_F+\frac{\beta}{4}\left \|Y^\top MY-I_p\right\|_F^2\\
        \overset{(i)}{\ge}&\left(\frac{\beta}{4}-\frac{3}{4}H_{X,\nabla f}C_X\right)\left \|Y^\top MY-I_p\right\|_F^2\ge\frac{\beta}{8}\|Y^\top MY-I_p\|_F^2
    \end{aligned}\]
    where the inequality $(i)$ follows from Proposition \ref{pr:dist}.
    The proof is complete.                                         
\end{proof}
The following theorem establishes the equivalence between the local minimizers of \ref{ocp} and \ref{slep}.
\begin{theorem}\label{th:locmr}
 Under Assumptions \ref{as:blanket} and \ref{as:s}, for any feasible point $X\in\mathcal{S}_M(n,p)$, suppose the penalty parameter $\beta\ge\beta_X$. 
Then, \ref{slep} and \ref{ocp} share the same local minimizers in $\mathbb{B}(X,\frac{\rho_X}{2}),X\in\mathcal{S}_M(n,p)$.
\end{theorem}

\begin{proof}
A local minimizer of \ref{slep} must be a second-order stationary point of \ref{slep}. Thus, by Proposition \ref{pr:info}, all minimizers of \ref{slep} in $\mathbb{B}(X,1)$ are feasible. Moreover, since $h(X) = f (X),\forall X\in\mathcal{S}_M(n,p)$, any local minimizer of \ref{slep} is a local minimizer of \ref{ocp}.

Conversely, if $X^*\in\mathbb{B}(X,\frac{\rho_X}{2})$ is a local minimizer of \ref{ocp}, then there exists a constant $\eta_X\le\frac{\rho_X}{2}$ such that $f(Z)\ge f(X^*)$ holds for every $Z\in\mathbb{B}(X,\eta_X)\cap\mathcal{S}_M(n,p)$. Let $0<\delta_X\le\frac{\eta_X}{2}$ be a constant such that $\|Z^\top MX-I_p\|_F\le\frac{\eta_X}{3C_X},\forall Z\in\mathbb{B}(X,\delta_X)$.
Then, for every $Z\in\mathbb{B}(X,\delta_X)$, Proposition \ref{pr:dist} yields $\|\mathcal{R}(Z)-X^*\|_F\le\|\mathcal{R}(Z)-Z\|_F+\|Z-X^*\|_F
    \le \frac{3}{2}\|Z\|_F\|Z^\top MZ-I_p\|_F+\frac{\eta_X}{2}
    <\eta_X$. Thus, $h(\mathcal{R}(Z))\ge h(X^*)$. By Proposition \ref{pr:varh},we then have
\[\begin{aligned}
    h(Z)-h(X^*)&=h(Z)-h(\mathcal{R}(Z))+h(\mathcal{R}(Z))-h(X^*)\ge\frac{\beta}{8}\|Z^\top MZ-I_p\|_F^2\ge0,
\end{aligned}\]
Thus, $X^*$ is a local minimizer of \ref{slep}. The proof is complete.
\end{proof}

\subsubsection{Equivalence on 
Łojasiewicz gradient inequality}\label{relogra}

In this subsection, we establish the Łojasiewicz gradient inequality of $h(X)$ based on the Riemannian Łojasiewicz gradient inequality for \ref{ocp}. The main results are summarized in the following theorem.
\begin{theorem}\label{th:lograr}
    Under Assumptions \ref{as:blanket} and \ref{as:s}, for any feasible point $X\in\mathcal{S}_M(n,p)$, suppose that the penalty parameter $\beta\ge\beta_X$, and that $f$ satisfies the Riemannian Łojasiewicz gradient inequality at $X\in\mathcal{S}_M(n,p)$ with Łojasiewicz exponent $\theta\in\left(0,1/2\right]$. Then $h$ also satisfies the Euclidean Łojasiewicz gradient inequality at $X$ with the same Łojasiewicz exponent $\theta$.
\end{theorem}
\begin{proof}
Since $f$ satisfies the Riemannian Łojasiewicz gradient inequality at $X\in\mathcal{S}_M(n,p)$ with Łojasiewicz exponent $\theta\in\left(0,1/2\right]$, there exist a constant $c>0$ and a neighborhood $\mathcal{U}_X$ of $X$ on $\mathcal{S}_M(n,p)$ such that 
  \begin{equation}\label{eq:loja_f}
    \left\|\grad  f (Y)\right\|_F \ge c|f(Y)-f(X)|^{1-\theta}, \forall Y\in\mathcal{U}_X.    
  \end{equation}
  Let $\bar{c}=c\cdot\frac{\lambda_{\min}(\bar{M})}{2\kappa^{\frac{1}{2}}(\bar{M})(1+\lambda_{\max}(\bar{M})C_X^2)}$. Then for
    every $Z\in\left\{Y\in\mathbb{B}(X,\rho_X):\left\|Y^\top MY-I_p\right\|_F^{1-2\theta}\le\frac{\lambda^\frac{1}{2}_{\min}(\bar{M})\beta^\theta}{16\bar{c}}\right\}$, we have
    \[\begin{aligned}
         &\|\nabla h(Z)\|_F
        \ge\frac{\lambda_{\min}(\bar{M})}{2\kappa^{\frac{1}{2}}(\bar{M})(1+\lambda_{\max}(\bar{M})C_X^2)}\|\grad  f(X)\|_F+\frac{\lambda_{\min}^{\frac{1}{2}}(\bar{M})\beta}{8}\left\|Z^\top MZ-I_p\right\|_F\\
        \overset{(i)}{\ge}&\bar{c}\left|f(\mathcal{R}(Z))-f(X)\right|^{1-\theta}+\frac{\lambda_{\min}^{\frac{1}{2}}(\bar{M})\beta}{8}\left\|Z^\top MZ-I_p\right\|_F\\
        \ge& \bar{c}\left|f(\mathcal{R}(Z))-f(X)\right|^{1-\theta}+\frac{\lambda_{\min}^{\frac{1}{2}}(\bar{M})\beta}{8}\left\|Z^\top MZ-I_p\right\|_F\\
         =& \bar{c}\left|h(\mathcal{R}(Z))-h(X)\right|^{1-\theta}+\frac{\lambda_{\min}^{\frac{1}{2}}(\bar{M})\beta}{8}\left\|Z^\top MZ-I_p\right\|_F\\
         \ge&\bar{c}\left(\left|h(Z)-h(X)\right|^{1-\theta}-\left|h(Z)-h(\mathcal{R}(Z))\right|^{1-\theta}\right)+\frac{\lambda_{\min}^{\frac{1}{2}}(\bar{M})\beta}{8}\left\|Z^\top MZ-I_p\right\|_F\\
          \overset{(ii)}{\ge}&\bar{c}\left(\left|h(Z)-h(X)\right|^{1-\theta}_F-2\beta^{1-\theta}\left\|Z^\top M Z-I_p\right\|_F^{2-2\theta}\right)+\frac{\lambda_{\min}^{\frac{1}{2}}(\bar{M})\beta}{8}\left\|Z^\top MZ-I_p\right\|_F\\
           \ge&\bar{c}\left|h(Z)-h(X)\right|^{1-\theta}_F+\beta^{1-\theta}\left\|Z^\top MZ-I_p\right\|_F\left(\frac{\lambda_{\min}^{\frac{1}{2}}(\bar{M})\beta^\theta}{8}-2\bar{c}\left\|Z^\top MZ-I_p\right\|_F^{1-2\theta}\right)\\
        \ge &\bar{c}\left|h(Z)-h(X)\right|^{1-\theta}_F. 
    \end{aligned}\]
    Here the inequality $(i)$ follows from the  Riemannian Łojasiewicz gradient inequality of the objective function $f$, and the inequality $(ii)$ follows from
    \[\begin{aligned}
        &\left|h(Z)-h(\mathcal{R}(Z))\right|^{1-\theta}\\
        =&\left|f(\mathcal{A})-f(\mathcal{R}(Z))+\frac{\beta}{4}\left\|Z^\top MZ-I_p\right\|_F^2\right|^{1-\theta}\\
        \le&\left|f(\mathcal{A}(Z))-f(\mathcal{R}(Z))\right|^{1-\theta}+\beta^{1-\theta}\left\|Z^\top M Z-I_p\right\|_F^{2-2\theta}\\
       \le &H_{X,\nabla f}^{1-\theta}\left\|\mathcal{A}(Z)-\mathcal{R}(Z)\right\|_F^{1-\theta}+\beta^{1-\theta}\left\|Z^\top M Z-I_p\right\|_F^{2-2\theta}\\       \le&\left(\left(H_{X,\nabla f}\|Z\|_F\right)^{1-\theta}+\beta^{1-\theta}\right)
       \left\|Z^\top M Z-I_p\right\|_F^{2-2\theta}\\
       \le& 2\beta^{1-\theta}\left\|Z^\top M Z-I_p\right\|_F^{2-2\theta}.
    \end{aligned}\]
    The proof is complete.
\end{proof}

\section{Applications}\label{algo}
The established equivalence between \ref{ocp} and \ref{slep} enables us to employ unconstrained optimization methods directly to solve \ref{ocp} through \ref{slep}. By applying efficient unconstrained optimization algorithms to \ref{slep}, we can derive efficient infeasible algorithms for \ref{ocp}. In this section, we introduce an infeasible framework, termed SLG, which applies gradient-based optimization methods to \ref{slep}.

The detailed framework is presented in Algorithm \ref{al:slg}. It is worth noting that since the iterates generated by SLG are not necessarily confined to $\Smnp$, the feasibility violation of the sequence $\{X_k\}$ is usually not very small, especially during the early stages of the algorithm. Therefore, when a solution with a low feasibility violation is required, a postprocessing procedure can be performed within the SLG framework, as described in Steps 6-12 of Algorithm \ref{al:slg}.

    \begin{algorithm}[H]\label{al:slg}
	\caption{Gradient-based method for \ref{slep} (SLG).}
	Choose the initial point $X_0$ and penalty parameter $\beta$;

 Construct the penalty function $h_\beta(X)$;

  Set $k=0$, $D_0=-\nabla h(X_0)$;
 
	\While{not terminated}{

   $X_{k+1}=X_k+\alpha_kD_k$;

   Update $D_k$ and set $k:=k+1$;
}

\eIf{high feasibility accuracy is required}{$\Tilde{X}:=\mathcal{R}(X_k)$;}{$\Tilde{X}:=X_k$;}
Return $\Tilde{X}$.
\end{algorithm}

%



    When we choose $D_k=-\nabla h_\beta(X_k)$ and fix $\alpha_k$ as a constant in Step 5 of Algorithm \ref{al:slg}, the SLG reduces to the gradient descent (GD) method for solving \ref{slep}. Furthermore, if $\alpha_k$ is chosen according to the Barzilai-Borwein (BB) step size rule \cite{barzilai_two_1988,dai_projected_2005}, as defined in \eqref{eq:abb}, then SLG corresponds to the Barzilai-Borwein gradient method for solving \ref{slep}. Additionally, by setting
$D_k=\mu_kD_{k-1}-\nabla h_\beta(X_k+\alpha_k\mu_kD_{k-1})$,
we obtain the Nesterov accelerated gradient (NAG) method \cite{nesterov_method_1983} for solving \ref{slep}. These methods also have Riemannian optimization counterparts \cite{absil_optimization_2008,iannazzo_r_2018,zhang_AnES_2018}.

In Table \ref{ta:comput}, we compare the per-iteration computational costs of these algorithms. Specifically, Table \ref{ta:comput} analyzes the computational complexity of key components commonly encountered in Riemannian gradient-based methods and in Algorithm \ref{al:slg}.
We also present a comparative assessment of the per-iteration computational costs for several gradient-based methods of SLEP and their corresponding Riemannian versions. For the Riemannian methods, we examine versions employing the inner product $\langle\cdot,\cdot\rangle_M$ with QR-based retraction —as implemented in the state-of-the-art solver ManOpt\cite{boumal_manopt_2014} —as well as versions using the canonical inner product $\langle\cdot,\cdot\rangle$ with polar-based retraction. It is important to note that the ManOpt version is applicable only when $M$ is of full rank. Terms in bold denote operations that cannot be parallelized.

As shown in Table \ref{ta:comput}, we can conclude that Riemannian optimization methods generally incur a higher computational cost per iteration compared to Algorithm \ref{al:slg}. One notable advantage of the proposed penalty method is its lower per-iteration computational cost compared to Riemannian optimization methods.  Moreover, Riemannian optimization methods involve $O(p^3)$ computations at each iteration, which are typically not parallelizable. These results further demonstrate that employing efficient unconstrained gradient-based approaches for \ref{slep} can lead to efficient optimization approaches for solving \ref{ocp}. 



 \begin{table}[htbp]
\tiny
	\centering 
	\setlength{\tabcolsep}{0.5mm}{\begin{tabular}{c|c|c|c|c|c|c}
		\hline\hline
		&\multicolumn{4}{c|}{\textbf{Oracles of Riemannian gradient-based methods}}&\multicolumn{2}{c}{\textbf{Oracles of Algorithm \ref{al:slg}}}\\
	\hline    \makecell{\textbf{Inner}\\\textbf{product}}&\makecell{$\langle\cdot,\cdot\rangle_M$}&$2n^2p+\mathcal{O}(np)$&\makecell{ $\langle\cdot,\cdot\rangle$}&$\mathcal{O}(np)$&\makecell{ $\langle\cdot,\cdot\rangle$}&$\mathcal{O}(np)$\\
       \hline
		\multirow{7}{*}{\makecell{\textbf{Riemannian}\\ \textbf{gradient}}}
		&$M^{-1}\nabla f(X)-X\Psi(X^\top \nabla f(X))$&\makecell{$1FO+\mathcal{O}(n^3 p)$\\$+4np^2$}&$\nabla f(X)-MXS$& \makecell{$1FO+2n^2p+$\\$5np^2+\mathcal{O}(p^3)$}&\multirow{7}{*}{\makecell{\textbf{Euclidean}\\\textbf{gradient}}} &\multirow{7}{*}{$1FO+2n ^2p+11np^2$}\\
    \cline{2-5}
    &$\nabla f(X)$&$1FO$&$\nabla f(X)$&1FO&&\\
		&$\nabla f(X)^\top X$&$2np^2$&MX &$2n^2p$&&\\
		&$X\Psi(\nabla f(X)^\top X)$&$2np^2$&$X^\top M^2 X$ &$np^2$ &&\\
		&$M^{-1}\nabla f(X)$&$\mathcal{O}(n^3p)$&$X^\top M\nabla f(X)$&$2np^2$& &\\
  &&&Equation for $S$&$\mathcal{O}(p^3)$& &\\
  &&&$MX S$&$2np^2$& &\\
		\hline
		\multirow{15}{*}{\makecell{ \textbf{Retraction}}}&\multicolumn{2}{c|}{\textbf{QR-based retraction}}& \multicolumn{2}{c|}{\textbf{Polar-based retraction}}&\multicolumn{2}{c}{\multirow{15}{*}{\textbf{No retraction}}}\\
  \cline{2-5}
  &\makecell{ $M^{-\frac{1}{2}}\operatorname{qf}(M^{\frac{1}{2}}(X+T))$}&\makecell{$2n^2p+\frac{4}{3}np^2$\\$+\bm{\mathcal{O}(p^3)}$}& \makecell{$(X+T)(I_p+T^\top MT)$}& \makecell{$2n^2p+5np^2$\\$+\bm{\mathcal{O}(p^3)}$}&\multicolumn{2}{c}{}\\
  \cline{2-5}
  	&$MT$&$2n^2p$ & 
   $MT$&$2n^2p$&\multicolumn{2}{c}{ }\\
		&$T^\top MT$&$np^2$ & $T^\top MT$&$np^2$&\multicolumn{2}{c}{}\\
		&\makecell{
		 $I_p+T^\top M T=L^\top L$}&$\bm{\mathcal{O}(p^3)}$& \makecell{$I_p+T^\top MT=VDV^\top$}&$\bm{\mathcal{O}(p^3)}$ &\multicolumn{2}{c}{ }\\
	  &$(X+T)L^{-1}$& $\frac{1}{3}np^2$&$(X+T)VD^{-\frac{1}{2}}V^\top$&$4np^2$&\multicolumn{2}{c}{ }\\
   \cline{2-5}
   & \multicolumn{2}{c|}{\textbf{Inverse of the QR-based retraction}}&\multicolumn{2}{c|}{\textbf{Inverse of the polar-based retraction}}&\multicolumn{2}{c}{ }\\
    \cline{2-5}
   & \makecell{$YR-X$} &\makecell{$2n^2p +4np^2$\\$+\mathcal{O}(p^4)$}& \makecell{$YZ-X$} &\makecell{$2n^2p +4np^2$\\$+\mathcal{O}(p^3)$}&\multicolumn{2}{c}{ }\\
    \cline{2-5}
  	&$MX$&$2n^2p$ & 
   $MX$&$2n^2p$&\multicolumn{2}{c}{ }\\
		&$Y^\top MX$&$2np^2$ & $Y^\top MX$&$2np^2$&\multicolumn{2}{c}{}\\
  & Equation for $R$&$\mathcal{O}(p^4)$ &Equation for $Z$&$\mathcal{O}(p^3)$&\multicolumn{2}{c}{}\\
  & $YR$&$2np^2$ & $YZ$&$2np^2$&\multicolumn{2}{c}{}\\
	  \hline\makecell{
	  \textbf{Vector}\\ \textbf{transport}} &$T-X\Psi(T^\top M X)$ &\makecell{
	  	$2n^2p$\\$+4np^2$}
& $T-MXS(T)$ &\makecell{
	  	$2n^2p+5np^2+$\\
	  	$\mathcal{O}(p^3)$
}& \multicolumn{2}{c}{\textbf{No vector transport}}\\
   \hline
   \multicolumn{7}{c}{\textbf{In total}}\\
   \hline
 \textbf{GD}&\multicolumn{2}{c|}{$1FO+\mathcal{O}(n^3p)+2n^2p+\frac{16}{3}np^2+\bm{\mathcal{O}(p^3)}$}& \multicolumn{2}{c|}{$1FO+4n^2p+10np^2+\bm{\mathcal{O}(p^3)}$}&\multicolumn{2}{c}{\multirow{3}{*}{$1FO+2n^2p+11np^2$}}\\
  \cline{1-5}
\textbf{BB}&\multicolumn{2}{c|}{$1FO+\mathcal{O}(n^3p)+8n^2p+\frac{28}{3}np^2+\bm{\mathcal{O}(p^3)}$}& \multicolumn{2}{c|}{$1FO+6n^2p+15np^2+\bm{\mathcal{O}(p^3)}$}&\multicolumn{2}{c}{}\\
  \cline{1-5}
\textbf{NAG}&\multicolumn{2}{c|}{$1FO+\mathcal{O}(n^3p)+10n^2p+16np^2+\mathcal{O}(p^4)+\bm{\mathcal{O}(p^3)}$}& \multicolumn{2}{c|}{$1FO+12n^2p+28np^2+\bm{\mathcal{O}(p^3)}$}&\multicolumn{2}{c}{}\\
	  \hline\hline
	\end{tabular}}
\caption{Comparison of the computational complexity of the first-order oracles between Algorithm \ref{al:slg} and its Riemannian counterparts.}
\label{ta:comput}
\end{table}








\section{Numerical experiments}\label{numerical}
In this section, we compare the numerical performance of Algorithm \ref{al:slg} with its corresponding Riemannian gradient-based counterparts on two classes of test problems: the quadratic minimization problem over the generalized Stiefel manifold (Test Problem 1) and the smoothing approximation problem of SGCCA (Test Problem 2). 
 All the numerical experiments in this section were conducted using MATLAB R2018a on an Ubuntu 18.10 operating system, running on a workstation equipped with an Intel(R) Xeon(R) Silver 4110 CPU at 2.10GHz	and 384.0 GB of RAM.
\subsection{Test problem}
In this subsection, we introduce test problems. The first is a quadratic minimization problem defined on the generalized Stiefel manifold:
\begin{problem}\label{t1}
    \[
		\begin{aligned}
 			\min_{X\in\mathbb{R}^{n\times p}} &f(X)=\frac{1}{2}\tr(X^\top AX)+\alpha\tr(G^\top X)\\
			\st\;\;\;&X^\top M X=I_p.
		\end{aligned}
\]
where the matrix $A\in\mathbb{R}^{n\times n}$ and the matrix $G\in\mathbb{R}^{n\times p}$.

\end{problem}
In Problem \ref{t2}, the matrix $A$ is defined as $A:=U^\top D U$,
where $U:=\operatorname{qf}(\operatorname{rand}(n,n))\in\mathbb{R}^{n\times n}$ and $D\in\mathbb{R}^{n\times n}$ is a diagonal matrix with elements $D(i,i):=\theta^{1-i}\;\;\;\text{for all }i=1,2,\dots,n$.
Here, $\theta\ge1$ is a parameter controlling the decay rate of the eigenvalues of $A$.
The matrix $G\in\mathbb{R}^{n\times p}$ is generated as $G:=QE$, where $\tilde{Q}=\operatorname{rand}(n,p)$, $Q\in\mathbb{R}^{n\times p}$ with normalized columns defined by $Q(:,i)=\tilde{Q}(:,i)/||\tilde{Q}(:,i)||_2(i=1,2,\dots,p)$, and 
$E\in\mathbb{R}^{p\times p}$ is a diagonal matrix with entries $E(i,i):=\eta^{i-1}\;\;\;\text{for all }i=1,2,\dots,p$.
The scalar $\alpha>0$ represents the scale between the quadratic and linear terms, while $\eta\ge1$ is a parameter that determines the growth rate of the column norms of $M$.
The matrix $M$ is an $n\times n$ sparse, positive semi-definite symmetric matrix generated by $M=\operatorname{sprandsym}(n,d,r),$
where $\operatorname{sprandsym}(n ,d,r)$ produces an $n\times n$ sparse matrix $M$ with approximately $d\cdot n^2$ nonzero elements and eigenvalues given by $r\in\mathbb{R}^n$. The elements of $r$ are uniformly sampled from $[0,1]$, i.e.,
$r(i)=\operatorname{rand}(1), 0\le i\le n$. Unless otherwise stated, the default values of the parameters in this problem are $\theta=1.01$, $\eta=1.01$, and $d=0.01$.
The initial point $X_0$ is randomly chosen on the manifold $\mathcal{S}_M(n,p)$ as $X_0=\mathcal{R}(\operatorname{rand}(n,p))$. 
We designed five groups of test problems, where only one parameter is varied at a time while keeping all others fixed. We describe the varying parameters of each group specifically as follows:
\begin{itemize}
    \item Number of rows of the variable: $n = 250j$ for $j = 1, 2, 3, 4, 5, 6, 7, 8$.
    \item Number of columns of the variable: $p = 20j$ for $j = 1, 2, 3, 4, 5, 6$.
    \item Decay of the eigenvalues of $A$: $\theta = 1.00 + 0.01j$ for $j = 1, 2, 3, 4, 5, 6$.
    \item Difference between column norms of $G$: $\eta = 1.00 + 0.01j$ for $j = 1, 2, 3, 4, 5, 6$.
    \item Dominance of the linear term: $\alpha = 10^{j}$ for $j= -2, -1, 0, 1, 2$.
\end{itemize} 

The second test problem is the smoothing problem of the $l_{2,1}$-norm SGCCA problem with $k=2$:
\begin{problem}\label{t2}
  \begin{equation}\label{eq:ssgca}
     \begin{aligned}	&&\min_{X\in\mathbb{R}^{n\times p}}\;\;\;&-\frac{1}{2}\tr \left(X^\top\mathbb{E}\left[\begin{pmatrix}
	\zeta_1\zeta_1^\top&\zeta_1\zeta_2^\top\\ 
\zeta_2\zeta_1^\top&\zeta_2\zeta_2^\top\\
   \end{pmatrix}\right] X\right)+\gamma \sum_{i=1}^{p} s_{\mu}(||X_{i\cdot}||_2),\\
	   &&\st\;\;\;& X^\top \mathbb{E}\left[\begin{pmatrix}
	\zeta_1\zeta_1^\top&0\\ 
0&\zeta_2\zeta_2^\top\\
   \end{pmatrix}\right]X=I_p,
	   \end{aligned}
\end{equation}
    where $s_\mu(t)$ is defined as in \eqref{smoothfun}, $\zeta_1\in\mathbb{R}^{n_1}$ and $\zeta_2\in\mathbb{R}^{n_2}$ are random vectors.
\end{problem}
In Problem \ref{t2}, we set $\mu=0.001$ and $\gamma=0.05$ by default. We tested the sample average approximation (SAA, \cite{arora_cca_2017}) problem using synthetic data. 
The exact covariance matrices are generated as follows \cite{chen_sparse_2013,gao_sparse_2017}:
 \[\mathbb{E}[\zeta_i\zeta_i^\top]:=\operatorname{sprandsym}(n_i,d_i,r_i), \;i=1,2,\]
 \[\mathbb{E}[\zeta_1\zeta_2^\top]:=\mathbb{E}[\zeta_1\zeta_1^\top]U\Lambda V^\top \mathbb{E}[\zeta_2\zeta_2^\top].\]
 where $\Lambda\in\mathbb{R}^{p\times p}$ is a diagonal matrix, and $U\in\mathbb{R}^{n_1\times p}, V\in\mathbb{R}^{n_2\times p}$ satisfy $\|U\|_{2,0}=s_1\ge p$, $\|V\|_{2,0}=s_2\ge p$, and 
$U^\top Cov(\mathcal{Z}_1,\mathcal{Z}_1) U=V^\top Cov(\mathcal{Z}_2,\mathcal{Z}_2) V=I_p$.
 By default, we set $d_1=d_2=0.01$. The elements $r_i,\;i=1,2$, and the diagonal elements of $\Lambda$ are generated uniformly on $[0,1]$. 
Suppose $Z_1\in\mathbb{R}^{m\times n_1}$ and $Z_2\in\mathbb{R}^{m\times n_2}$ are the centralized data matrices of $\zeta_1$ and $\zeta_2$ from the distribution $\mathcal{N}\left(\begin{pmatrix}
       0\\
       0
   \end{pmatrix},\mathbb{E}\left[\begin{pmatrix}
	\zeta_1\zeta_1^\top&\zeta_1\zeta_2^\top\\ 
\zeta_2\zeta_1^\top&\zeta_2\zeta_2^\top\\
   \end{pmatrix}\right]\right)$, 
where $m$ denotes the sample size. We set $m=10(n_1+n_2)$.  In the SAA formulation, the covariance matrices are approximated as 
   \[\mathbb{E}[\zeta_i\zeta_j^\top]\approx T_{\sigma}(\frac{1}{m}Z_iZ_j^\top),\]
where the thresholding operator $T_\sigma$ improves the approximation of sparse real covariance matrices:
 \[T_\sigma(M)_{ij}=\left\{\begin{aligned}
	   &M_{ij},\;\;\;&\text{if } |M_{ij}|>\sigma,\\ 
	   &0,\;\;\;&\text{if } |M_{ij}|\le\sigma,
   \end{aligned}\right.\]
where $\sigma=\sqrt{\frac{\operatorname{log}n}{m}}$. 
We considered four groups of test problems, each varying one parameter while keeping the others fixed. By default, we set $\mu=0.001$ and $\gamma=0.05$.
\begin{itemize}
    \item Dimension of random variables: $n_1=n_2=250j$, $j=1,2,3,4,5,6,7,8$.
\item Number of pairs of canonical variables: $p=20+5j$, $j=0,1,2,3,4,5$.
\item Sparsity parameter: $\gamma=0.02j$, $j=1,2,3,4,5$.
\item Smoothing parameter: $\mu=0.0005j$, $j=1,2,3,4,5$.
\end{itemize}
\subsection{Implementation details}


We compared the performance of Algorithm \ref{al:slg} with the BB gradient setting to that of the Riemannian BB gradient method. Specifically, in Algorithm \ref{al:slg}, we set $D_k=-\nabla h_\beta(X)$ and compute $\alpha_k$ using the alternative BB (ABB) strategy:
\begin{equation}\label{eq:abb}
   \alpha^{ABB}_{k}=\left\{\begin{aligned}
       &&\frac{\left|\left\langle S_{k},Y_{k} \right\rangle\right| }{\left\langle Y_{k},Y_{k} \right\rangle }\;\;\;&\text{if k is even,}\\
        &&\frac{\left\langle S_{k},S_{k} \right\rangle }{\left|\left\langle S_{k},Y_{k} \right\rangle \right|}\;\;\;&\text{if k is odd.}
   \end{aligned}\right.
\end{equation}
where $S_k=X_{k}-X_{k-1}$, $Y_k=\nabla h(X_{k})-\nabla h(X_{k-1})$, and $\alpha_0=0.001$. For brevity, we refer to this algorithm as SLBB. Following the discussions in \cite{gao_parallelizable_2019,xiao_solving_2021}, the penalty parameter is selected as $\beta= 0.1 \|\nabla f(X_0)\|$ in our numerical experiments. We then terminate the algorithm whenever 
\begin{equation}\label{eq:stop}
    \norm{\nabla h(X)}\le \varepsilon
\end{equation}
or when the number of iterations reaches $N$. 

For fairness, we compare the numerical performance of SLBB with that of the Riemannian gradient method using BB stepsizes \cite{boumal_manopt_2014}. The Manopt package employs a canonical metric \cite{shustin_preconditioned_2021} for optimization problems over the generalized Stiefel manifold. Therefore, we first tested the Riemannian BB method with this canonical metric (ManBB) for comparison. Moreover, we evaluated the numerical performance of the Riemannian BB method under the Euclidean metric (RBB). Both ManBB and RBB were terminated when the maximum number of iterations exceeded $N$ or when the norm of its Riemannian gradient is smaller than $\epsilon$, that is, $\|\grad f (X_k)\|_F\le \epsilon$. 

Here, we set $\epsilon=10^{-4}$, $N=2000, 10000$ for Problems \ref{t1} and \ref{t2}, respectively.
Moreover, all compared algorithms were initialized at the same randomly generated initial point $X_0 \in \Smnp$. 
\subsection{Numerical results}

The numerical results for Problem \ref{t1} are listed in Table 3.
From the first five groups of results, we make the following observations:  In the vast majority of cases, all solvers converged to the same function value when initialized from identical initial points. Moreover, they exhibited comparable KKT violations, typically on the order of $10^{-5}$. Notably, SLBB consistently outperformed the other solvers in terms of CPU time, owing to its lower per-iteration computational cost.

To provide a more detailed perspective, we measured the time consumed by each component of the algorithms (over 200 iterations), as shown in Figure \ref{timepart}. The calculation of the Riemannian gradient in Riemannian algorithms requires significantly more time than that of the Euclidean gradient in our algorithm. In addition, the computation of retraction and transport mapping occupies a significant portion of the total runtime. Furthermore, ManBB requires a substantial amount of time for the inner product calculations.

The last result group shows that as the rank-deficient degree of $M$ increases, SLBB remains nearly unaffected in terms of accuracy, although its iteration count increases. We also examined the variation of stationary violations over iterations, as shown in Figure \ref{rank}. It is evident that RBB tends to diverge as the degree of rank deficiency increases.

\begin{table}[htbp]
\footnotesize
			\centering
			 \setlength{\tabcolsep}{1mm}{\begin{tabular}{|c|c|c|c|c|c|c|c|c|c|c|c|c|}
				\hline
	\makecell{Test \\problems}&\multicolumn{3}{c|}{Function values}&\multicolumn{3}{c|}{Stationarity}&\multicolumn{3}{c|}{Iterations}&\multicolumn{3}{c|}{CPU time(s)}\\
			\hline
&ManBB&RBB&SLBB&ManBB&RBB&SLBB&ManBB&RBB&SLBB&ManBB&RBB&SLBB\\
\hline
$n$&&&&&&&&&&&&\\
\hline
250& -22.94& -22.94&  -22.94&  9.23e-05&  7.65e-05&  7.63e-05& 748.80&  125.40&  187.50& 3.74&  0.61&  \textbf{0.24}  \\ \hline
500&-34.07&  -34.07&  -34.07&  8.78e-05&  7.98e-05&  6.88e-05& 213.20&  114.90&  129.50& 1.56&  0.70&  \textbf{0.24}  \\ \hline
750&-38.92 & -38.92 & -38.92 & 7.57e-05 & 7.45e-05 & 7.83e-05 & 127.60 & 148.10 & 148.90 & 1.31 & 1.13 & \textbf{0.37}  \\ \hline
1000&-42.24&   -42.24&   -42.24&   8.18e-05&   8.22e-05&   8.56e-05&  240.70&   195.60 &  186.70&  3.12&   1.59&   \textbf{0.47} \\ \hline
1250&-44.59&   -44.59 &  -44.59&   8.44e-05 &  8.80e-05 &  7.65e-05&  271.50&   247.00 &  202.50&  4.48&   2.34&   \textbf{0.63}  \\ \hline
1500&-44.36&   -44.36&   -44.36&   8.60e-05&   7.71e-05&   7.47e-05&  202.40&   250.40  & 216.80&  5.25  & 2.74 &  \textbf{0.82}  \\ \hline
1750&-49.19&   -49.19&   -49.19&   8.41e-05&   8.68e-05&   8.29e-05 & 173.30&   314.30&  249.20 & 5.38 & 4.04&   \textbf{1.18}  \\ \hline
2000&-48.02&   -48.02  & -48.02 &  7.16e-05&   8.49e-05  & 7.25e-05 & 169.10 &  318.70  & 251.50 & 6.60&   4.81&   \textbf{1.45}  \\ \hline
$p$&&&&&&&&&&&&\\
\hline
20&-21.81&  -21.81&  -21.81&  7.36e-05 & 7.25e-05&  7.16e-05& 104.10&  144.90 & 133.30 &0.86  &0.57  &\textbf{0.23}  \\ \hline
40&-36.77 & -36.77 & -36.77 & 6.11e-05 & 8.35e-05 & 7.94e-05 & 245.80 & 195.30 & 165.40 & 2.89 & 1.37 & \textbf{0.37}  \\ \hline
60&-48.36 & -48.36 & -48.36 & 8.45e-05 &  7.89e-05 & 7.92e-05 &229.30&  237.90 & 200.60& 3.25  &2.33 & \textbf{0.56}  \\ \hline
80&-56.51 & -56.51 & -56.51&  8.37e-05 & 7.93e-05&  7.01e-05& 223.20&  236.90 & 215.20 &3.74  &2.99 & \textbf{0.75}  \\ \hline
100&-61.80  &-61.80  &-61.80  &9.16e-05&  8.23e-05 & 7.43e-05& 240.20&  236.90 & 226.50& 4.70&  4.18 & \textbf{1.00}  \\ \hline
120&-67.00 & -67.00&  -67.00 & 8.08e-05&  8.54e-05 & 6.93e-05& 286.80&  245.70 & 246.50 &6.36 & 5.58&  \textbf{1.27}  \\ \hline
$\theta$&&&&&&&&&&&&\\
\hline
1.01&-43.59&  -43.59 &  -43.59 &  7.95e-05 &  7.09e-05  & 6.46e-05 & 395.60 &  216.90 &  194.40&  5.19 &  1.78  & \textbf{0.48}  \\ \hline
1.04&-50.39 &  -50.39 &  -50.39 &  8.37e-05 &  7.89e-05 &  8.72e-05 &  87.90 &  404.60  & 338.50&  1.14 &  3.28&   \textbf{0.84}  \\ \hline
1.07&-55.65&   -55.65 &  -55.65&   6.78e-05&   8.75e-05&   8.55e-05&  120.70&   599.80&   450.80 & 1.57  & 4.86&   \textbf{1.11}  \\ \hline
1.10&-55.27 &  -55.27&   -55.27&   5.10e-05&   9.35e-05&   8.93e-05&  128.30 &  643.00&   456.20&  1.66 &  5.20 &  \textbf{1.13}  \\ \hline
1.13&-54.34 &  -54.34 &  -54.34 &  8.20e-05 &  8.84e-05  & 9.30e-05&   61.70 &  546.40 &  454.00&  0.80 &  4.42 &  \textbf{1.13}  \\ \hline
1.16&-57.58  & -57.58&   -57.58 &  6.59e-05&   8.73e-05&   9.17e-05&   67.60 &  613.60 &  506.30 & 0.88&   4.96 &  \textbf{1.26}  \\ \hline
$\eta$&&&&&&&&&&&&\\
\hline
1.01&-43.71 & -43.71 & -43.71 & 8.16e-05 & 8.18e-05 & 7.46e-05 & 158.50 & 196.40 & 183.70 & 2.07 & 1.61 & \textbf{0.48}  \\ \hline
1.04&-24.06&  -24.06 & -24.06 & 7.55e-05&  8.07e-05&  6.83e-05& 444.70&  229.00&  207.50& 5.94& 1.89&  \textbf{0.53}  \\ \hline
1.07&-15.91&  -15.91&  -15.91&  8.30e-05&  8.92e-05&  9.31e-05& 776.50&  278.80&  322.50& 10.38&  2.31&  \textbf{0.83}  \\ \hline
1.10&-11.82&  -11.82&  -11.82&  9.44e-05&  9.62e-05&  9.53e-05& 1558.30&  463.80&  535.00 &20.80&  3.84&  \textbf{1.39}  \\ \hline
1.13&-9.83&  -9.83&  -9.83&  9.61e-05&  9.86e-05&  9.89e-05& 3365.00&  931.70  &1095.00 &43.77&  7.55&  \textbf{2.82}  \\ \hline
1.16&-8.32&  -8.32&  -8.32&  1.63e-03&  9.96e-05&  9.99e-05& 9981.30&  1412.60&  1671.50& 128.20&  11.43&  \textbf{4.26}  \\ \hline
$\alpha$&&&&&&&&&&&&\\
\hline
0.01&-0.23 & -0.23 & -0.23 & 3.31e-02 & 9.56e-05 & 9.69e-05 & 2000.00 & 308.60 & 327.10 & 33.94 & 2.84 & \textbf{0.91}\\
 \hline
0.1&-3.22 & -3.22 & -3.22 & 8.21e-05 & 7.98e-05 & 7.98e-05 & 680.30 & 123.90 & 123.50 & 11.62 & 1.11 & \textbf{0.34}\\
 \hline 
1&-41.98 & -41.98 & -46.18 & 8.56e-05 & 8.07e-05 & 6.69e-05 & 162.70 & 196.70 & 199.30 & 2.76 & 1.76 & \textbf{0.53}\\
 \hline 
10&-579.54 & -579.54 & \textbf{-1995.41} & 7.64e-05 & 7.95e-05 & 7.98e-05 & 174.40 & 793.50 & 775.80 & 2.99 & 7.08 &\textbf{2.05} \\
 \hline
100&-3858.99 & -7275.48 & \textbf{-7276.40 }& 1.57e+03 & 7.20e-01 & \textbf{5.93e-04} & 2000.00 & 2000.00 & 1965.80 & 33.80 & 17.89 & \textbf{5.35}
\\
 \hline
 $c$&&&&&&&&&&&&\\
\hline
1&-&-41.83& -41.83 &-& 8.64e-05 & 7.24e-05 &-& 211& 175 &-& 1.98 & \textbf{0.52}  \\
\hline
0.98&-&-48.50& -48.50&-& 9.57e-05 & 9.73e-05 &-& 275& 202&-& 2.55& \textbf{0.51}
 \\
\hline
0.96&-&-51.19& -51.19 &-& 8.66e-05 & 6.80e-05&-& 511& 322&-& 4.77& \textbf{0.89}
\\
\hline
0.94&-&-54.73& -54.73 &-& 6.70e-05 & 9.52e-05 &-& 1813& \textbf{263}&-& 16.44 & \textbf{0.80}\\
\hline
0.92&-&-59.59& -68.23&-& 1.30e+00 & \textbf{6.43e-05} &-& 2000& \textbf{420} &-& 17.84 & \textbf{1.09} 
 \\
\hline
0.90&-&-62.16 & -81.42&-& 1.39e+00 & \textbf{9.30e-05} &-& 2000& \textbf{386}&-& 17.91& \textbf{1.13}\\
\hline
			\end{tabular}}
   \caption{Comparison between RBB and SLBB on problem \ref{t1} (Each test is run 10 times and averaged).
}
		\end{table}
  \begin{figure}[H]
        \centering
    \includegraphics[scale=0.55]{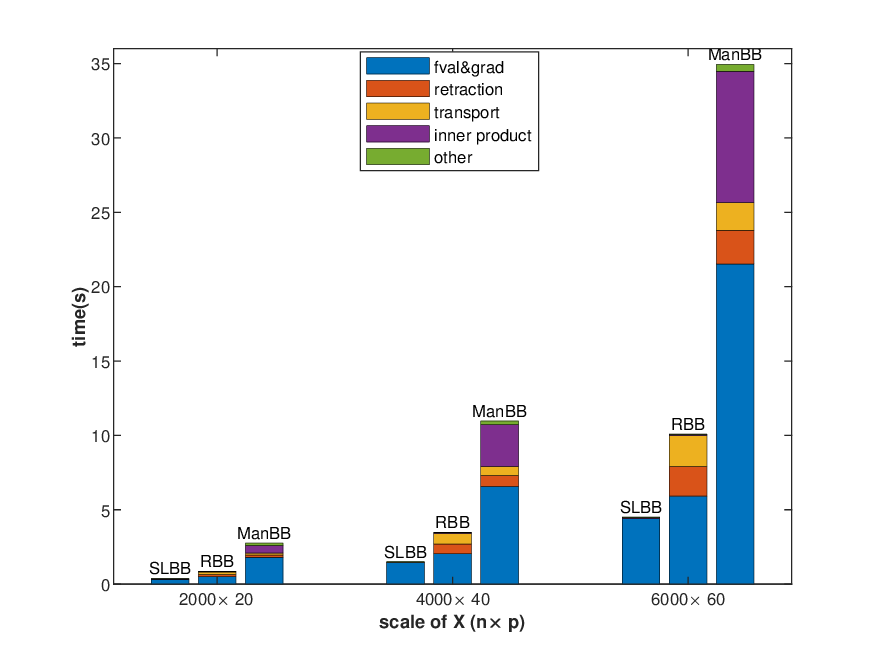}
        \caption{Comparison of the CPU time in each part with different problem scales.}
        \label{timepart}
    \end{figure}

\begin{figure}[H]
   		\centering
		\includegraphics[width=8cm,height=6cm]{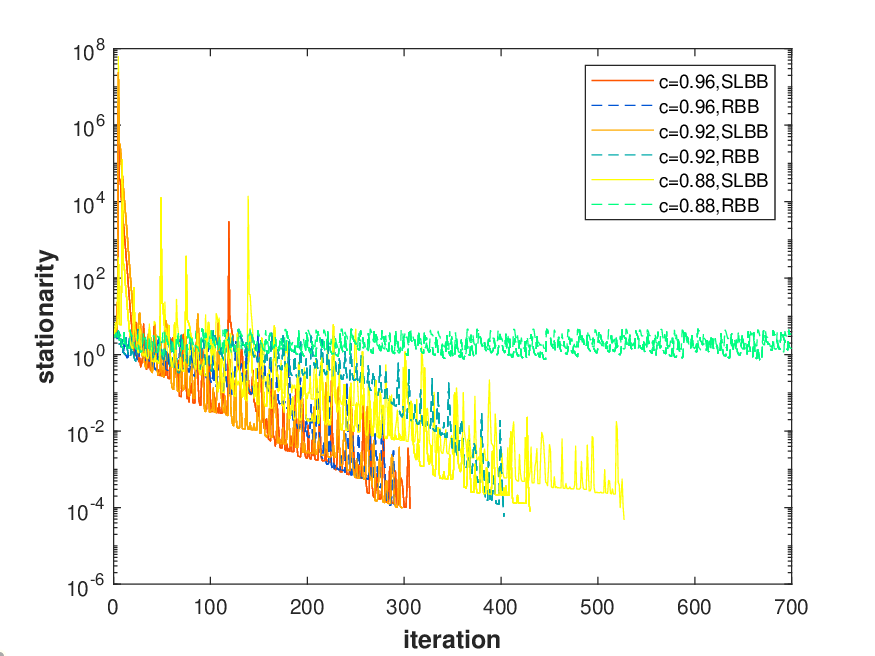}
  \caption{Comparison with varying $M$ rank $c$. 
}
\label{rank}
 	\end{figure}

In the following section, we present the numerical results for Problem \ref{t2} in Table \ref{t2r}. It is worth noting that across all experiments, SLBB consistently exhibited the most efficient computational performance in terms of CPU time.
 \begin{table}[htbp]
		\footnotesize
			\centering
			 \setlength{\tabcolsep}{1mm}{\begin{tabular}{|c|c|c|c|c|c|c|c|c|c|c|c|c|}
				\hline
				\makecell{Test \\problems}&\multicolumn{3}{c|}{Function values}&\multicolumn{3}{c|}{Stationarity}&\multicolumn{3}{c|}{Iterations}&\multicolumn{3}{c|}{CPU time(s)}\\
			\hline
$n_1=n_2$&&&&&&&&&&&&\\
\hline
200&-18.07 & -18.07 & -18.07 & 9.23e-05 & 9.96e-05 & 9.95e-05 & 3507.00 & 7250.70 & 7365.70 & 39.57 &  42.07 & \textbf{22.17} \\
 \hline 
400&-17.78 & -17.78 & -17.78 & 9.25e-05 & 9.95e-05 & 9.95e-05 & 2010.20 & 3023.80 & 3035.00 & 50.36 &  32.88 & \textbf{20.46}
\\
 \hline 
600&-18.46 & -18.46 & -18.46 & 8.71e-05 & 9.82e-05 & 9.82e-05 & 706.10 & 992.50 & 1017.00 & 32.20 &  18.65 & \textbf{12.81}\\
 \hline 
800&-18.01 & -18.01 & -18.01 & 9.63e-05 & 9.96e-05 & 9.96e-05 & 1484.30 & 3157.80 & 3185.60 & 112.13 &  90.98 & \textbf{63.54}
\\
 \hline 
1000&-17.82 & -17.82 & -17.82 & 9.01e-05 & 9.97e-05 & 9.97e-05 & 1724.30 & 4526.60 & 4497.40 & 185.78 &  192.01 & \textbf{135.36}\\
 \hline
 $p$&&&&&&&&&&&&\\
\hline
20&-12.03 & -12.03 & -12.03 & 9.33e-05 & 9.87e-05 & 9.92e-05 & 1258.00 & 2227.90 & 2251.80 & 35.19 &  22.21 & \textbf{15.11} \\
 \hline 
25&-15.62 & -15.62 & -15.62 & 8.97e-05 & 9.88e-05 & 9.94e-05 & 839.90 & 1324.40 & 1362.60 & 28.80 &  16.77 & \textbf{11.62} 
\\
 \hline 
30&-17.92 & -17.92 & -17.92 & 9.10e-05 & 9.89e-05 & 9.92e-05 & 1521.70 & 2722.20 & 2793.20 & 52.54 &  39.48 & \textbf{29.25} 
\\
 \hline 
35&-21.20 & -21.20 & -21.20 & 9.74e-05 & 9.93e-05 & 9.94e-05 & 1099.30 & 1781.70 & 1801.20 & 43.24 &  29.69 & \textbf{19.67} 
\\
 \hline 
40&-24.49 & -24.49 & -24.49 & 8.98e-05 & 9.91e-05 & 9.94e-05 & 1169.20 & 1793.30 & 1828.80 & 47.59 &  32.69 & \textbf{21.25}\\
 \hline
 $\gamma$&&&&&&&&&&&&\\
 \hline
0.02&-19.90 & -19.90 & -19.90 & 9.30e-05 & 9.98e-05 & 9.97e-05 & 889.90 & 2616.50 & 3016.80 & 31.21 &  38.34 & \textbf{28.85} \\
 \hline 
0.04&-18.93 & -18.93 & -18.93 & 9.31e-05 & 9.97e-05 & 9.97e-05 & 1980.80 & 3341.90 & 3328.40 & 67.59 &  47.87 & \textbf{30.47} 
\\
 \hline 
0.06&-17.68 & -17.68 & -17.68 & 9.40e-05 & 9.90e-05 & 9.92e-05 & 1182.60 & 1792.70 & 1842.40 & 40.94 &  26.01 & \textbf{17.21}\\
 \hline 
0.08&-16.79 & -16.79 & -16.79 & 8.78e-05 & 9.97e-05 & 9.93e-05 & 1668.10 & 2915.40 & 2625.60 & 56.64 &  43.28 & \textbf{25.65} 
\\
 \hline 
0.10&-16.37 & -16.37 & -16.37 & \textbf{8.98e-05} & 2.93e-03 & 4.78e-03 & \textbf{5944.60} & 10000.00 & 10000.00 & 203.09 &  140.68 & 90.39 
\\
 \hline
 $\mu$&&&&&&&&&&&&\\
\hline
0.0005&-18.25 & -18.25 & -18.25 & 9.50e-05 & 9.93e-05 & 9.89e-05 & 1315.40 & 1816.10 & 1845.30 & 44.54 &  25.83 & \textbf{16.96} 
\\
 \hline 
0.0010&-17.68 & -17.68 & -17.68 & 9.62e-05 & 9.88e-05 & 9.91e-05 & 1082.20 & 1737.20 & 1761.30 & 36.84 &  25.05 & \textbf{16.19}
\\
 \hline 
0.0015&-18.33 & -18.33 & -18.33 & 9.50e-05 & 9.93e-05 & 9.94e-05 & 1939.60 & 3900.00 & 3883.40 & 66.63 &  55.76 & \textbf{36.11} \\
 \hline 
0.0020&-18.25 & -18.25 & -18.25 & 9.62e-05 & 9.97e-05 & 9.96e-05 & 807.50 & 1634.20 & 1682.90 & 27.93 &  23.89 & \textbf{15.98} 
\\
 \hline 
0.0025&-17.89 & -17.89 & -17.89 & 9.20e-05 & 9.96e-05 & 9.92e-05 & 624.10 & 1048.70 & 1080.20 & 21.34 &  15.13 & \textbf{10.06} \\
 \hline
			\end{tabular}}
   \caption{Comparison between RBB and SLBB on problem \ref{t2} (Each test is run 10 times and averaged).
}
 \label{t2r}
\end{table}

\section{Conclusion}\label{conclu}

In this study, we developed a novel penalty model \ref{slep} to solve the optimization problem on the generalized Stiefel manifold with a possibly rank-deficient $M$. We prove the equivalence between the original problem \ref{ocp} and the penalty model \ref{slep} in the sense that, within a neighborhood of the feasible region and for a sufficiently large but finite penalty parameter $\beta$, they share the same first- and second-order stationary points. Based on the proposed penalty model \ref{slep}, we propose an infeasible algorithm (Algorithm \ref{al:slg}) for solving problem \ref{ocp}, which employs the gradient descent method with BB stepsizes. Extensive numerical experiments demonstrate that Algorithm \ref{al:slg} achieves superior performance compared with existing state-of-the-art Riemannian optimization methods.  

\paragraph{Acknowledgement}
The work of Xin Liu was supported in part by the National Natural Science Foundation of China (12125108, 12021001, 12288201), and RGC grant JLFS/P-501/24 for the CAS AMSS–PolyU Joint Laboratory in Applied Mathematics.

\bibliography{refer}
\end{document}